\documentclass[reqno,10pt]{amsart}
\usepackage{amsmath,amssymb,latexsym}
\usepackage{hyperref}
\usepackage{amsthm}
\usepackage{color}
\usepackage{tikz-cd}  
\usepackage[all]{xy}
\usepackage{tikz-cd}
\usepackage{seqsplit}
\allowdisplaybreaks

\hypersetup{
    colorlinks=true, %set true if you want colored links
    linktoc=all,     %set to all if you want both sections and subsections linked
    linkcolor=blue,  %choose some color if you want links to stand out
}

\newcommand{\Z}{{\mathbb{Z}}}
\newcommand{\Q}{{\mathbb{Q}}} 
\newcommand{\QQ}{{\overline{\mathbb{Q}}}}
    
\newcommand{\R}{{\mathbb{R}}}    
\newcommand{\C}{{\mathbb{C}}}

\newcommand{\p}{{\mathfrak{p}}}

\newcommand{\FF}{{\mathcal{F}}}

\newcommand{\OO}{{\mathcal{O}}}

\newcommand{\lra}{\longrightarrow}

\newcommand{\Gal}{\mathrm{Gal}}
\newcommand{\Ker}{\mathrm{Ker}}

\newcommand{\tor}{\mathrm{tor}}

\newcommand{\Hom}{\mathrm{Hom}}

\newcommand{\La}{\Lambda}
\newcommand{\gl}{\mathrm{Gl}}
\newcommand{\Dim}{\mathrm{Dim}}

\theoremstyle{plain}
\newtheorem{theorem}{Theorem}[section]
\newtheorem*{theorem*}{Theorem}

\newtheorem{thm}{Theorem}

\newtheorem{proposition}[theorem]{Proposition}
\newtheorem{rem}[theorem]{Remark}
\newtheorem{lemma}[theorem]{Lemma}
\newtheorem{corollary}[theorem]{Corollary}
\newtheorem{defn}[theorem]{Definition}

\newtheorem{exmp}{Example}

\begin{document}
\title{On characteristic ideal of Selmer group associated to Artin representations} 
\author{Dipramit Majumdar \& Subhasis Panda} 

\begin{abstract} { Selmer group for an Artin representation over totally real fields was studied by Greenberg \cite{gr2, gr} and Vatsal \cite{va}. In this paper we study the Selmer groups for an Artin representation over a totally complex field. We establish an algebraic function of the characteristic ideal of the Selmer group associated to Artin representation over the cyclotomic $\Z_p$- extension of the rational numbers under certain mild hypotheses and construct several examples to illustrate our result. We also prove that in this situation $\mu$-invariant of the dual Selmer group is independent of the choice of the lattice. 
 }\end{abstract}

\maketitle

\tableofcontents

\section{Introduction}

 The study of Iwasawa theory of  Artin representation of $\Delta:= \Gal(K/\Q)$, where $K$ is a totally real Galois extension of $\Q$, was pioneered by the work of Greenberg \cite{gr2, gr}. On the other hand, if $K$ is totally complex, the study of Selmer groups associated to Artin representation becomes more subtle. In a recent work, Greenberg-Vatsal \cite{gv}, studied  Iwasawa theory of  Artin representation of $\Gal(K/\Q)$, where $K$ is a totally complex Galois extension of $\Q$, under certain {\bf Hypothesis A}, which we explain as follows.\\
We fix an odd prime $p$. Let $K$ be a finite Galois extension of $\Q$ such that $p \nmid [K:\Q]$. If $\mu_p \not \subset K$, we can take $K'=K(\mu_p)$, then $K'$ is Galois over $\Q$ and $p \nmid [K':\Q]$, here $\mu_p$ denote the set of $p$-th roots of unity. As a consequence, from now on, we assume that $\mu_p \subset K$. We remark that the theory becomes quite complicated if $p \mid [K: \Q]$. The authors hope to study that situation in the future.\\
We fix an embedding $\iota_{\infty}$ of a fixed algebraic closure $\QQ$ of $\Q$ inside $\C$ and for each prime $l$ an embedding $\iota_l$ of $\QQ$ into a fixed algebraic closure $\QQ_l$ of $\Q_l$.\\
By an Artin representation of $\Delta$, we denote a finite dimensional irreducible representation $\rho: \Delta = \Gal(K/\Q) \to \gl_\FF(V)$, here $\FF$ is a finite extension of $\Q_p$ and $V$ is a $\FF$-vector space of dimension  $d(\rho)$. Let $\OO$ denote the ring of integer of $\FF$ and $\pi$ denote an uniformizer of $\OO$.\\
For an Archimedian place $\nu$ of $K$, let $K_\nu$ denote the completion of $K$ at $\nu$, which can be identified with $\R$ or $\C$. By $\Delta_\nu$, we denote the Galois group $\Gal(K_\nu/\R)$, which we identify with a subgroup of $\Delta$. Let $d_\nu^+(\rho)$ denote the multiplicity of the trivial representation of $\Delta_\nu$ in $\rho|_{\Delta_\nu}$. 
Since $K$ is Galois over $\Q$, $K$ is either totally real or totally complex. If $K$ is totally real field, then $d_\nu^+(\rho) = \dim_\FF(\rho)$, as $\Gal(K_\nu/\R) = \{ 1 \}$.
If $K$ is totally complex and $\nu$, $\nu'$ are two different places of $K$, then $\Delta_\nu$ and $\Delta_{\nu'}$ are conjugate in $\Gal(K/\Q)$. Therefore, the multiplicity of the trivial representation  in $\rho|_{\Delta_\nu}$ and $\rho|_{\Delta_{\nu'}}$ is same and $d_\nu^+(\rho)$ is independent of the choice of $\nu$. Hence simply denote $d_\nu^+(\rho)$ by $d^+(\rho)$. 
Further, we define $d^-(\rho):= d(\rho)-d^+(\rho)$.\\
Greenberg and Vatsal \cite{gv}, studied the situation when $d^+(\rho) =1$. Note that $K$ is totally real implies that $d^+(\rho) =d(\rho)$, which was studied in \cite{gr2, gr}. In this paper, we study Selmer group of Artin representations $\rho$, which satisfies $ 1 \le d^+(\rho) \le d(\rho)-1$. Note that this automatically implies that $d(\rho) \ge 2$ and hence $\rho$ is a non-trivial representation.\\
Observe that the embedding $\iota_p$ induces a prime in $K$ above $p$, which we will denote by $\wp$. By $\Delta_{\wp}$, we denote the decomposition group of the prime $\wp$ in $K$ above $p$. Assume that $V$ has a $\Delta_{\wp}$ invariant subspace $V_{\wp}^+$ of dimension $d^+(\rho)$ such that $V_{\wp}^+$ and $V_{\wp}^-:= V/V_{\wp}^+$ have no common irreducible component as a $\Delta_{\wp}$ representation. 
 
 To summarize, we have assumed the following hypothesis on the triple $(\rho,V,V_{\wp}^+):$
\begin{enumerate}
  \item[{\bf HYP 1:}]  $p$ is a fixed odd prime and $\Q(\mu_p) \subset K$ be a finite Galois extension of $\Q$ with $p \nmid |\Delta|=|\Gal(K/\Q)|$.
  \item[{\bf HYP 2:}]  The representation $\rho: \Delta \to \gl_\FF(V)$ is irreducible of $\Dim_\FF(V)= d(\rho)$, where $\FF$ a finite extension of $\Q_p$ and  the triple satisfies the following assumptions:
  \begin{enumerate}
  \item[{\bf HYP 2a:}]For the prime $\wp$ in $K$ above $p$ induced by $\iota_p$, $V$ has a $\Delta_{\wp}$-invariant subspace $V_{\wp}^+$ of dimension $d^+(\rho)$ with $0 < d^+(\rho) < d(\rho) $.
\item [{\bf HYP 2b:}] $V_{\wp}^+$ and $V_{\wp}^-:=V/V_{\wp}^+$ have no common irreducible component as $\Delta_{\wp}$ representations.
  \end{enumerate}
\end{enumerate}

If $p$ is an ordinary prime for a compatible system of $\ell$-adic representation $V=\{ V_\ell \}$ over $\Q$, then $\Dim(V_p^-) = d^-$ (see \cite[Eqn (43)]{gr4}). With this analogy in mind, we choose a subspace $V_\wp^+$ of dimension $d^+$. We remark that, unlike the classical case, in the case of Artin representations, the choice of $d^+$ dimensional subspace is not canonical. Moreover, the Selmer group does depend on the choice of this subspace. The condition that $V_\wp^+$ and $V_\wp^-$ have no common irreducible components is analogues to the classical condition that the representation is $p$-distinguished.\\
Finally, we remark that $d^+(\rho) < d(\rho)$ implies that $K$ is a totally complex, which necessarily implies that $[K:\Q]$ is even.\\

From now on, we assume that our triple $(\rho, V, V_\wp^+)$ satisfy {\bf Hyp 1 and 2}. We choose an $\OO$-lattice $T$ of $V$, which is invariant under the action of $\Delta$. The $\Delta_{\wp}$-filtration on $V$ induces a filtration $0 \subset T_{\wp} \subset T$ of $T$. We denote by $A$ (respectively, $A_{\wp}^+$, respectively, $A_{\wp}^-$) the discrete $\OO$-module $V/T$ (respectively, the image of $V_{\wp}^+$ in $A$, respectively, $A/A_{\wp}^+$) which is isomorphic to $(\FF/\OO)^{d}$ (respectively, $(\FF/\OO)^{d^+}$, respectively, $(\FF/\OO)^{d^-}$) as an $\OO$-module. Following Greenberg-Vatsal \cite{gv}, we define $Sel^{GV}_{\Q}(A,A_\wp^+)$ as
\begin{eqnarray*}
 Sel^{GV}_{\Q}(A,A_\wp^+) & = \Ker\Big(H^1(\Q_\Sigma/\Q, A) \lra \underset { \ell \in \Sigma, \ell \nmid p\infty}{\oplus} \,  H^1(I_\ell , A) \bigoplus \, H^1(I_{p}, A_\wp^-)\Big), 
 \end{eqnarray*}
 here $\Sigma$ denotes a finite set of primes containing $p, \infty$ and the primes which are ramified in $K/\Q$ and for a finite prime $\omega \in \Sigma$, $I_{\omega}$ denotes the inertia group at $\omega$ with respect to a fixed prime $\overline{\omega}$ of $\QQ$ over $\omega$.\\
 
The cyclotomic $\Z_p$-extension of $\Q$ is denoted by $\Q_\infty$ with $\Gamma:=\Gal(\Q_\infty/\Q) \cong \Z_p$. By $\Lambda_\OO$, we denote the Iwasawa algebra $\OO[[ \Gamma ]]$, which is non-canonically isomorphic to the ring of formal power series $\OO[[X]]$. In section 2, we define the Selmer group over the cyclotomic $\Z_p$ extension $\Q_\infty$, which we will denote by $Sel^{GV}_{\Q_\infty}(A,A_\wp^+)$. We will denote the Pontryagin dual of $Sel^{GV}_{\Q_\infty}(A,A_\wp^+)$ by $X^{GV}_{\Q_\infty}(T,T_\wp^+)$. As a $\Lambda_\OO$-module, $X^{GV}_{\Q_\infty}(T,T_\wp^+)$ is finitely generated (See Lemma \ref{fingen}).\\

For a finitely generated torsion $\Lambda_\OO$-module M,  by $C_{\Lambda_\OO}(M)$, we denote the $\Lambda_\OO$ characteristic ideal of $M$. Thus, to talk about the characteristic ideal of  Selmer group associated to an Artin representation, we need the following assumption.\\

{\bf HYP Tor{$_\rho$}:} The Pontryagin dual of the Selmer group $X^{GV}_{\Q_\infty}(T,T_\wp^+)$ associated to $(\rho,V,V_{\wp}^+)$ satisfying {\bf Hyp 1 \& 2} is a finitely generated $\Lambda_\OO$-torsion module.\\

For a finitely generated torsion $\Lambda_\OO$-module $M$, by $\mu(M)$, we denote the $\mu$-invariant of $M$. Recall that $\mu(M)$ is the largest integer $n$ such that $C_{\Lambda_\OO}(M) \subset \pi^n \Lambda_\OO$. In section $5$, we study how $\mu$-invariant of the dual Selmer group  $X^{GV}_{\Q_\infty}(T,T_\wp^+)$ depends on the choice of the lattice $T$. In case of $p$-adic representations, it is known that the $\mu$-invariant does depend on the choice lattice. For example, Mazur \cite[Section 10]{ma}, showed that for the prime $p=5$, the $5$-isogeny class of $X_0(11)$ has three elliptic curves (namely $11.a1, 11.a2$ and $11.a3$ in LMFDB label) and their $\mu$-invariants at $5$ are different. In case of Artin representation over totally real fields, it was shown in ( \cite[Proposition 2.4]{gr}) that the $\mu$-invariant is independent of the choice of lattices. We show the situation is similar in our case. 
 
 \begin{thm}
 Let $(\rho,V,V_\wp^+)$ be a triple which satisfies {\bf HYP 1, 2 \& Tor{$_\rho$}}. Further assume that $H^0(\Delta_\wp,V)=0$. Then the $\mu$-invariant of the dual Selmer group  $X^{GV}_{\Q_\infty}(T,T_\wp^+)$ is independent of the choice of the lattice.
 \end{thm}
 
Theorem A easily follows from Theorem \ref{muconstant}. It'll be interesting to know whether $\mu$-invariant is $0$ in the above situation.\\
 
Remaining of the paper is devoted to the study of `algebraic functional equation' of the dual Selmer group  $X^{GV}_{\Q_\infty}(T,T_\wp^+)$. To motivate this, let $V_f = \{V_{f,\ell} \}$ be a compatible system of $\ell$-adic representation of $G_\Q:=\Gal(\QQ/\Q)$ associated to a $p$-ordinary newform of weight $2$ with rational Fourier coefficient. Let $V_f^\ast =  \{ V_{f, \ell}^\ast \} = \{ Hom_{\Q_\ell}(V_{f, \ell} , \Q_\ell(1)) \}$. To $V_f$ and $V_f^\ast$, we can associate complex $L$-functions $L_{V_f}( s)$ and $L_{V_f^\ast}(s)$, which can be completed and extended to the whole complex plane, and they are related by an `analytic functional equation' $\Lambda_{V_f}( s) = \pm \Lambda_{V_f^\ast}(2- s)$. The twisted $L$-function of $V_f$ similarly satisfy a functional equation  relating $L_{V_f}(s, \chi)$ and $L_{V_f^\ast}( 2-s, \chi^{-1})$ for any character $\chi$ of $\Gamma = \Gal(\QQ/\Q)$. By work of Mazur-Swinnerton-Dyer, Manin, Visik, Amice-Velu \cite[Theorem V.7.9]{be2}, under a certain hypothesis, we can associate an `analytic' $p$-adic $L$-function $L^p_{V_f}$ to $V_f$ and moreover it satisfies interpolation property $L^p_{V_f}(\chi) = c_\chi L_{V_f}(1, \chi)$ for some constant $c_\chi \in \C $. As a consequence, we see that, under certain hypothesis,  `analytic' $p$-adic $L$-function of $V_f$ satisfies a functional equation 
$$L^p_{V_f}(\chi) = u_\chi L^p_{V_f^\ast}(\chi^{-1}),$$
here $u_\chi$ is an element of $\C_p^\times$ of absolute value $1$. By the cyclotomic Iwasawa main conjecture for $f$, we have $C_{\Z_p[[\Gamma]]}(X(T_f/\Q_\infty)) = (L^p_{V_f}(X))$ as ideals in the Iwasawa algebra $\Z_p[[X]]$, here $L^p_{V_f}(X)$ denotes the power series corresponding to $L^p_{V_f}$. As a consequence of the `analytic' functional equation of the $p$-adic $L$-function of $f$, we obtain the following equality of ideals in the Iwasawa algebra $\Z_p[[X]]$,
 \begin{equation}\label{algfndefn}
 C_{\Z_p[[\Gamma]]}(X(T_f/\Q_\infty)) = C_{\Z_p[[\Gamma]]}(X(T_f^\ast/\Q_\infty)^\iota),
 \end{equation}
 here for a $\Z_p[[\Gamma]]$-module $M$, by $M^\iota$ we understand the module whose underlying abelian group is $M$ but the action of $\gamma \in \Gamma$ on $m \in M^\iota$ is by $\gamma^{-1}$. The above equation \eqref{algfndefn} is the `algebraic functional equation'. One expects `algebraic functional equation' to be true purely from an algebraic perspective without the knowledge of `analytic' $p$-adic $L$-function or Iwasawa main conjecture. For $p$-ordinary $\ell$-adic representation $V$ of dimension $d$, Greenberg \cite{gr4} and Perrin-Riou \cite{pr}, have proved `algebraic functional equation' under certain hypothesis. In particular, `algebraic functional equation' (over cyclotomic extension) of modular forms $f$ of weight at least $2$ is known under certain hypothesis. This result has been extended in two different ways. This result has been extended in the case of big  $\Lambda$-adic representation associated to Hida's $p$-ordinary (resp. nearly $p$-ordinary) family of cuspform \cite{jp} (resp. Hilbert cusp form \cite{jm}). On the other hand, this result has also been extended to the Selmer group associated to $2$-dimensional motives satisfying certain hypothesis over arbitrary $p$-adic Lie extensions \cite{jo} instead of cyclotomic extension.\\

In this paper, we also study the `algebraic functional equation' for Artin representations. Note that the Galois representation associated to a weight $1$ cusp form is an odd irreducible Artin representation of dimension $2$, which in our notation corresponds to $d=2$ and $d^+ = 1$. Hence our theory includes weight $1$ cusp forms and their symmetric powers. Unlike the case of $p$-adic representation, in the case of Artin representation, $V^\ast$ is not an Artin representation. Let  $(\sigma, U)$ be the $\Delta$-representation $\Hom_\FF(V, \FF(\tau))$, here $\tau$ is the Teichm\"{u}ller character. We identify $V^\ast \cong U(\kappa)$, here $\kappa: \Gamma \to 1+ p\Z_p$ is an isomorphism coming from the action of  $\Gal(K_{\infty}/K)$ on $\mu_{p^{\infty}}$, where $\mu_{p^{\infty}}$ is the group of $p$-power roots of unity. Following Rubin, for a character $\varphi: \Gamma \to \OO^*$, we define an $\OO$-linear automorphism $Tw_{\varphi} : \Lambda_\OO \to \Lambda_\OO$ induced by the map $Tw_{\varphi}(\gamma)= \varphi(\gamma)\gamma$. For a finitely generated torsion $\Lambda_\OO$-module $M$, by $Tw_{\varphi}(M)$, we understand the module whose underlying group is $M$, but the action of $\gamma \in \Gamma$ is by $Tw_{\varphi}(\gamma)$. With this notation, we prove the following `algebraic functional equation'.

\begin{thm}
Let $(\rho,V,V_\wp^+)$ be a triple which satisfies {\bf HYP 1, 2 \& Tor{$_\rho$}} and $T$ be an $\OO$-lattice of $V$ invariant under action of $\Delta$. Let $(\sigma,U)$ denote the Artin representation $U:= \Hom_{\FF}(V, \FF(\tau))$, where $\tau$ denotes the Teichmuller character and $T^\prime$ the corresponding lattice in $U$. Assume that the triple $(\sigma,U,U_\wp^+)$ satisfies {\bf HYP Tor{$_\sigma$}}. Further assume that the $\Delta_\wp$-representations $V$ and $U$ does not contain trivial sub-representation. Then as ideals of ${\Lambda_\OO}$
\begin{equation*}
  C_{\Lambda_\OO}(X^{GV}_{\Q_\infty}(T,T{_\wp^+})^\iota)=  Tw_{\kappa} (C_{\Lambda_\OO} (X^{GV}_{\Q_\infty}(T',(T'){_\wp^+})).
\end{equation*}

\end{thm}

In a recent work, J. Lee \cite[Theorem D]{le} proved a stronger version of the `algebraic' functional equation for Abelian varieties with good reduction at $p$. He showed that, in that situation, such equality holds not only for the characteristic ideal but also for the corresponding elementary modules. As a consequence, he does not require torsion assumption. The second name author hopes to study a similar question in the case of Artin representations.\\

The structure of the article is as follows. In section 2, we define various Selmer groups involved. A crucial component of our proof of functional equation is the generalized Casseles-Tate pairing of Flach. Unfortunately, the local condition above $p$ of the Selmer group $Sel^{GV}_{\Q_n}(A,A_\wp^+)$ does not satisfy the necessary conditions required for the generalized Casseles-Tate pairing of Flach, and thus following Rubin, we introduced an auxiliary Selmer group $Sel^{g}_{\Q_n}(A,A_\wp^+)$. In section 3, we compare the Selmer groups $Sel^{GV}_{\Q_n}(A,A_\wp^+)$ and $Sel^{g}_{\Q_n}(A,A_\wp^+)$ (see Theorem \ref{2.1}). In section 4, we prove control theorems for Artin representations (see Theorem \ref{3.1}, Theorem \ref{3.4}). In section 5, we study the dependence of $\mu$-invariant under a choice of lattice and prove Theorem A. In section 6, we prove algebraic functional equation under certain mild hypothesis. In section 7, we provide some examples of Artin representations that satisfy all the hypothesis of Theorem B.\\

{\bf Acknowledgment: } It is a pleasure to acknowledge several e-mail communication, advice, and re-mark of Professor Somnath Jha. The first author was partially supported by IIT Madras ERP project RF/2021/0658/MA/RFER/008794 and MHRD SPARC grant number 445. The second author was partially supported by IIT Madras ERP project RF/2021/0658/MA/RFER/008794.

\section{Definition of Selmer Group}\label{section1}

Let $(\rho, V, V_\wp^+)$ be a triple satisfying {\bf Hyp 1 \& 2}. Let us choose an $\OO$-lattice $T$ of $V$, which is invariant under the action of $\Delta$. The $\Delta_{\wp}$-filtration on $V$ induces a filtration $0 \subset T_{\wp} \subset T$ of $T$. We denote by $A$ (respectively, $A_{\wp}^+$, respectively, $A_{\wp}^-$) the discrete $\OO$-module $V/T$ (respectively, the image of $V_{\wp}^+$ in $A$, respectively, $A/A_{\wp}^+$) which is isomorphic to $(\FF/\OO)^{d}$ (respectively, $(\FF/\OO)^{d^+}$, respectively, $(\FF/\OO)^{d^-}$) as an $\OO$-module.

Let $\Sigma$ be a finite set of primes containing $p$, $\infty$ and the primes which are ramified in $K/\Q$ and let $\Q_{\Sigma}$ be the maximal unramified extension of $\Q$ outside of $\Sigma$. Since all the primes ramified in $K/\Q$ are in $\Sigma$, we have $K \subset \Q_{\Sigma}$. By $\Q_{\infty}$, we denote the cyclotomic $\Z_p$ extension of $\Q$ and by $\Gamma$, the Galois group of $\Q_{\infty}$ over $\Q$ that is, $\Gamma=\Gal(\Q_\infty/\Q)$, which is canonically isomorphic to $\Z_p$. 
Let $\Gamma_n=\Gamma^{p^n}$ and $\Q_n = \Q_\infty^{\Gamma_n}$. Hence  $\Gal(\Q_\infty/\Q_n) \cong \Gamma_n$ and $\Gal(\Q_n/\Q) \cong \Gamma/\Gamma_n \cong \Z/p^n\Z$. \\
%From now on, we fix a $\Delta_\wp$-invariant subspace $V_{\wp}^+$ of $V$ satisfying {\bf HYP 2 }. 
The notation $\omega \mid \Sigma$ means a prime $\omega$ in $\Q_n$ lying above some prime in $\Sigma$. For $\omega$, a finite prime in $\Q_n$, let $G_{\omega}$ denotes the decomposition group at $\omega$ and $I_{\omega}$ denotes the inertia group at $\omega$ with respect to a fixed prime $\overline{\omega}$ of $\QQ$ over $\omega$. Similarly, $v_n$ (respectively, $v_\infty$) denote a prime in $\Q_n$ (respectively, $\Q_\infty$) such that $v_\infty \mid v_n \mid \Sigma$ for all $n \ge 0$. By $G_{v_n}$ (respectively, $G_{v_\infty}$), we denote the decomposition subgroup of $\Gal(\QQ/\Q_n)$ for $\overline{v}/v_n$ (respectively, $\Gal(\QQ/\Q_\infty)$ for $\overline{v}/v_\infty$), and $I_{v_n}$ (respectively, $I_{v_\infty}$) denote the inertia subgroup of $G_{v_n}$ (respectively, $G_{v_\infty}$). Note that as $p$ is totally ramified in $\Q_n$ (resp. $\Q_\infty$), the unique prime above $p$ in $\Q_n$ (resp. in $\Q_\infty$) will be denoted by $\p_n$ (resp. by $\p_\infty$). 

The Selmer group associated to the representation $(\rho,V,V_{\wp}^+)$ over $\Q_n$ is a subgroup of $H^1(\Q_\Sigma/\Q_n, V)$  defined by the Selmer structure ( see \cite[Definition 2.2]{be}) $\mathcal{L} = (\mathcal{L}_\omega)$  given by:
\begin{equation*}
\mathcal{L}_\omega:= H^1_g(G_\omega,V) = \begin{cases}
                                   \Ker(H^1(G_\omega,V) \to H^1(I_\omega,V) ) \text{ if } \omega  \nmid p, \\
                                   \Ker(H^1(G_{\p_n}, V) \to H^1(I_{\p_n}, V_\wp^-))  \text{  if } \omega= \p_n.
                                   \end{cases} 
\end{equation*}
We remark that we denote the local condition by $H^1_g$ as for Selmer group of $p$-adic representations the Block-Kato \cite[Equation (3.7.1), (3.7.2)]{bk} local conditions $H^1_g$ matches with the conditions as mentioned above \cite[Lemma 2]{fl}. The projection map $pr: V \to A$ induces a map in cohomology $pr^*: H^1(G_\omega,V) \to H^1(G_\omega,A)$. Let $H^1_g(G_\omega,A)$ be the image of $H^1_g(G_\omega,V)$ under the map $pr^*$, where $\omega$ is a prime in $\Q_n$ lying above some prime in $\Sigma$. 
Following ( see \cite[Definition 5.1]{ru}), we define a Selmer group of $A$ over $\Q_n$ as follows:
\begin{defn}\label{selg}
\begin{equation*}
Sel^g_{\Q_n}(A,A_\wp^+) := \mathrm{ker}\Big(H^1(\Q_\Sigma/\Q_n, A) \lra \underset { \omega \mid \Sigma, \, \omega \nmid p\infty}{\oplus} \, \frac{ H^1(G_\omega, A)}{H^1_g(G_\omega,A)} \underset { \omega = \p_n}{\oplus} \, \frac{ H^1(G_{\p_n}, A)}{H^1_g(G_{\p_n},A)} \Big).
\end{equation*}
\end{defn}

On the other hand following Greenberg-Vatsal ( See \cite[Introduction, page 3]{gv}), we define

\begin{defn}\label{selgv}
\begin{equation*}
Sel^{GV}_{\Q_n}(A,A_\wp^+) := \mathrm{ker}\Big(H^1(\Q_\Sigma/\Q_n, A) \lra \underset { \omega \mid \Sigma, \omega \nmid p\infty}{\oplus} \, \frac{ H^1(G_\omega, A)}{H^1_{GV}(G_\omega,A)} \underset { \omega= \p_n}{\oplus} \, \frac{ H^1(G_{\p_n}, A)}{H^1_{GV}(G_{\p_n},A)} \Big),
\end{equation*}
\end{defn}
where
\begin{equation*}
H^1_{GV}(G_\omega,A) = \begin{cases}
                                   \Ker(H^1(G_\omega,A) \to H^1(I_\omega,A) ) \text{ if } \omega \nmid p, \\ 
                                   \Ker(H^1(G_{\p_n}, A) \to H^1(I_{\p_n}, A_\wp^-))  \text{  if } \omega = \p_n. 
                                   \end{cases} 
\end{equation*}
Equivalently, $Sel^{GV}_{\Q_n}(A,A_\wp^+)$ can be defined as
\begin{eqnarray*}
 Sel^{GV}_{\Q_n}(A,A_\wp^+) & = \Ker\Big(H^1(\Q_\Sigma/\Q_n, A) \lra \underset { \omega \mid \Sigma, \omega \nmid p\infty}{\oplus} \,  H^1(I_\omega , A) \underset { \omega = \p_n}{\oplus} \, H^1(I_{\p_n}, A_\wp^-)\Big).
 \end{eqnarray*}

Since our prime p is odd, the 1-cocycle classes are trivial at the Archimedian primes of $\Q_n$. Therefore, the direct sum in definitions \ref{selg} and \ref{selgv} are just over the finite primes in $\Q_n$. 
We remark that when $d^+(\rho)=1$, the definition \ref{selgv} matches with the definition of Selmer group as in ( See \cite[Introduction, page 3]{gv}).\\

For $Sel^\dagger_{\Q_n}(A,A_\wp^+)$ with $\dagger \in \{ g, GV \}$, the Pontryagin dual is defined as
\begin{equation*}
X^\dagger_{\Q_n}(T,T_\wp^+) = (Sel^\dagger_{\Q_n}(A,A_\wp^+))^\vee := \Hom_{\mathrm{cont}}(Sel^\dagger_{\Q_n}(A,A_\wp^+), \Q_p/\Z_p ).
\end{equation*}
For the infinite extension $\Q_{\infty}$ of $\Q$, the Selmer group $Sel^\dagger_{\Q_\infty}(A,A_\wp^+)$ is defined to be the inductive limit of $Sel^\dagger_{\Q_n}(A,A_\wp^+)$ with respect to the natural restriction maps and the Pontryagin dual $X^\dagger_{\Q_\infty}(T,T_\wp^+)$ is defined to be the projective limit of $X^\dagger_{\Q_n}(T,T_\wp^+)$ with respect to the natural corestriction maps. The group $\Gamma = \Gal(\Q_\infty/\Q)$ has a natural $\OO$-linear action on $Sel^\dagger_{\Q_\infty}(A,A_\wp^+)$. 
Therefore, we consider $Sel^\dagger_{\Q_\infty}(A,A_\wp^+)$ as a discrete $\Lambda_\OO$-module and $X^\dagger_{\Q_\infty}(T,T_\wp^+)$ as a compact $\Lambda_\OO$ module, where $\Lambda_\OO$ denotes the Iwasawa algebra $\OO[[ \Gamma ]]$.

\begin{lemma}\label{fingen}
$X^\dagger_{\Q_\infty}(T,T_\wp^+)$ is a finitely generated $\La_\OO$-module.
\end{lemma}

\begin{proof}
This follows from the fact that $H^1(\Q_\Sigma/\Q_\infty,A)$ is co-finitely generated $\Lambda_\OO$-module. 
\end{proof}

Conjecturally one expects that $X^\dagger_{\Q_\infty}(T,T_\wp^+)$ is a torsion $\La_\OO$-module provided {\bf HYP 1 \& 2} are satisfied. For example, $X^\dagger_{\Q_\infty}(T,T_\wp^+)$  is torsion $\La_\OO$-module if $d^+(\rho)=1$ and $\rho$ satisfies {\bf HYP 1 \& 2} ( See \cite[Proposition 4.2]{gv}).\\

Let $K_{\infty} = K \Q_{\infty}$ be the cyclotomic $\Z_p$ extension of $K$. Since $(p,[K:\Q])=1$; we have $K \cap \Q_\infty=\Q$. Therefore, we can identify $\Gal(K_{\infty}/K)$ with $\Gal(\Q_{\infty}/\Q)$ by the usual restriction map. Since $\mu_p \subset K$, we have $K_\infty=K(\mu_{p^{\infty}})$, where $\mu_{p^{\infty}}$ denotes the group of $p$-power roots of unity. Using the action of  $\Gal(K_{\infty}/K)$ on $\mu_{p^{\infty}}$, we get an isomorphism $\kappa: \Gamma \to 1+ p\Z_p$. By $\tau$, we denote the Teichm\"{u}ller character. The action of $G_\Q:= \Gal(\QQ/\Q)$ on $\mu_{p^{\infty}}$ is given by $\tau\kappa$.

Following Greenberg-Vatsal ( See \cite[Proposition 4.7]{gv}), we extend the definition of Selmer groups of Artin representations to include the Tate twists of Artin representations. 
Define $T^*= \Hom (A, \mu_{p^{\infty}})$, $V^*= T^*\otimes_{\OO} \FF$ and $A^*= V^*/T^*$. 
The representation $V^*$ of $G_{\Q}$ is not an Artin representation. Observe that
\begin{equation*}
V^* = \Hom (A, \mu_{p^{\infty}})\otimes_{\OO} \FF \cong \Hom_{\FF}(V, \FF(\tau \kappa)) \cong \Hom_\FF(V, \FF(\tau)) \otimes_\FF \FF(\kappa), 
\end{equation*}
as a $G_{\Q}$ representation, where $\FF(\tau \kappa)$ and $\FF(\kappa)$ represent two one dimensional vector spaces over $\FF$ on which $G_{\Q}$ acts through $\tau \kappa$ and $\kappa$, respectively. Let $U$ denote the $G_\Q$ representation $U:= \Hom_\FF(V, \FF(\tau))$, then $V^* \cong U \otimes_{\FF} \FF(\kappa)$. Note that $U$ is an Artin representation of $\Delta$, which we'll denote by $\sigma$. Let $T'$ denote the $\OO$-lattice of $U$ induced from $T$ and define $A':=U/T'$.
Let $U_\wp^+$ denote the orthogonal complement of $V_\wp^+$ under the pairing between $V$ and $U$. 
This gives us a $\Delta_\wp$ filtration $0 \subset U_\wp^+ \subset U $. 
Let $d^+(\sigma)$ denotes the dimension of $U_\wp^+$. By using the orthogonality relation between $V_\wp^+$ and $U_\wp^+$ we have $\dim U_\wp^+ + \dim V_\wp^+ = \dim V$. 
Therefore, $\dim U_\wp^+ = \dim V - \dim V_\wp^+ = d(\rho) - d^+(\rho)= d^-(\rho)=d^+(\sigma)$. 
Since $V^* \cong U \otimes \FF(\kappa)$, we get a $\Delta_\wp$ filtration $0 \subset U_\wp^+ \otimes_{\FF} \FF(\kappa) := (V^*)_\wp^+ \subset V^*$ from the filtration of $U$. This gives us a filtration for $A^*$. 
As in the case of Artin representation $\rho$, using the lattice and filtration of $A^*$, we can define Selmer groups and its Pontryagin duals for $A^*$ over $\Q_n$ and $\Q_\infty$.\\

Note that as $U_\wp^+$ is the orthogonal complement of $V_\wp^+$ under the pairing between $V$ and $U$, and since $V_\wp^+$ and $V_\wp^-$ have no common irreducible components ({\bf HYP 2b}), $\Hom(V_\wp^-,\FF(\tau)) \subset U$ is the orthogonal complement of $V_\wp^+$. Therefore, $U_\wp^+$ and $\Hom(V_\wp^-,\FF(\tau))$ are isomorphic as $G_{\Q}$ representations. As a consequence we see that $(V^*)_\wp^+ := U_\wp^+ \otimes \FF(\kappa)$ is isomorphic as $G_{\Q}$ representation to $\Hom(V_\wp^-,\FF(\tau)) \otimes \FF(\kappa) := (V_\wp^-)^*$. Similarly we obtain $(V_\wp^+)^{*} \cong (V^*)_\wp^-$ as $G_{\Q}$ representations.\\

The character $\kappa$ factors through $\Gal(\Q_{\infty}/\Q)$, therefore $V^* \cong U$ as a $G_{\Q_{\infty}} : = \Gal(\QQ/\Q_\infty)$ representation. It is easy to show that for $\dagger \in \{ g, GV \}$
\begin{equation*}
Sel^{\dagger}_{\Q_\infty}(A^*,(A^*){_\wp^+}) \cong Sel^\dagger_{\Q_\infty}(A',(A'){_\wp^+}) \otimes_\OO \OO(\kappa) \text{  and } X^\dagger_{\Q_\infty}(T^*,(T^*){_\wp^+}) \cong X^\dagger_{\Q_\infty}(T',(T'){_\wp^+}) \otimes_\OO \OO(\kappa^{-1}).
\end{equation*}
 %as an action of $\Gamma$, where $\dagger \in \{ g, GV \}$. \\

We further extend the definition of Selmer group to include a twist of Artin representations by characters of $\Gamma$ ( See \cite[\S 1.2, page 408]{jm}). 
For a continuous group homomorphism $\varphi: \Gamma \to \OO^\times$, let $\FF(\varphi)$ (respectively, $\OO(\varphi)$) represents a  Galois module whose underlying group is $\FF$ (respectively, $\OO$) on which $G_\Q$ acts through $\varphi$.  
For an Artin representation $ ( \rho,V,V_\wp^+ )$, let $\rho(\varphi)$ denotes the $\FF$-representation space $V \otimes_{\FF} \FF (\varphi)$ with the diagonal action of $G_\Q$.
We define $T(\varphi):= T \otimes_\OO \OO(\varphi)$ and $A(\varphi):= A \otimes_\OO \OO(\varphi)$ with the diagonal $G_\Q$ action. Then we have a $\Delta_\wp$ filtration $0 \subset A_\wp^+(\varphi) \subset A(\varphi)$ of $A(\varphi)$ by applying $- \otimes \OO(\varphi)$ to the filtration of $A$ . Similarly, we can define $V^*(\varphi), T^*(\varphi), A^*(\varphi)$ and a filtration $(A^*)_\wp^+(\varphi)$. Using these, we can define the Selmer group $ Sel^{\dagger}_{L}(B,B_\wp^+,\varphi)$  and its Pontryagin dual $X^\dagger_{L}(C,C{_\wp^+},\varphi^{-1})$, where $\dagger \in \{ g, GV \}$, $L \in \{ \Q_n , \Q_\infty \}$, $B \in \{ A , A^* \}$ and $C \in \{ T , T^* \}$. Since $\varphi$ factors through $\Gamma$, as before we have
\begin{equation*}
Sel^{\dagger}_{\Q_\infty}(B,B_\wp^+,\varphi) \cong Sel^{\dagger}_{\Q_\infty}(B,B_\wp^+) \otimes_\OO \OO(\varphi), \text{  and }  X^\dagger_{\Q_\infty}(C,C{_\wp^+},\varphi) \cong X^\dagger_{\Q_\infty}(C,C{_\wp^+}) \otimes_\OO \OO(\varphi^{-1}).
\end{equation*}

\section{Comparison of  Selmer Groups} 

In this section, we study how the Selmer groups $Sel^{GV}_{\Q_n}(A,A_\wp^+, \varphi)$ and $Sel^{g}_{\Q_n}(A,A_\wp^+, \varphi)$ are related for any continuous group homomorphism $\varphi: \Gamma \to \OO^\times$. We show that under mild hypothesis,  $Sel^{g}_{\Q_n}(A,A_\wp^+, \varphi)$ is a subgroup of  $Sel^{GV}_{\Q_n}(A,A_\wp^+, \varphi)$ and moreover the quotient is finite whose cardinality uniformly bounded independent of $n$. A similar result holds for $A^*$. In the case of the Selmer group associated to $p$-adic representations, results comparing the Bloch-Kato Selmer group and Greenberg Selmer group over $\Q_{\infty}$ can be found in \cite{fl, oc}.

%Let us fix an odd prime $p$, a number field $K$ and an Artin representaion $\rho$ such that the 3-tuple $(p, K, \rho)$ satisfies the Hypothesis 1 and 2. In this section we compare the Selmer groups $Sel^{GV}_{L}(B,B_\wp^+,\varphi)$ and $ Sel^{g}_{L}(B,B_\wp^+,\varphi)$ defined in the previous section \ref{section1}, here $L \in \{ \Q_n , \Q_\infty \}$ and $B \in \{ A , A^* \}$. \\

%First we compare the Selmer groups $Sel^{GV}_{L}(A,A_\wp^+,\varphi)$ and $ Sel^{g}_{L}(A,A_\wp^+,\varphi)$. Recall that, we assume there exists a $\Delta_\wp$ filtration $0 \to V_\wp^+ \to V \to V_\wp^-:= V/V_\wp^+ \to 0$, where $V_\wp^+$ is a $\Delta_\wp$ sub-representation of $V$ of dimension $d^+(\rho)$ and $V_\wp^+$ and $V_\wp^-$ has no common irreducible sub-representations of $\Delta_\wp$. This filtration gives us a filtration $0 \to V_\wp^+(\varphi) \to V(\varphi) \to V_\wp^-(\varphi) \to 0$, which further induces filtration $0 \to A_\wp^+(\varphi) \to A(\varphi) \to A_\wp^-(\varphi) \to 0$.

Before proving our comparison theorems, we write down a few simple observations which will be useful throughout the paper. \\

Let $W$ be a $\Delta:= \Gal(K/\Q)$ representation, $T_W$ be a lattice in $W$, and $A_W =W/T_W$. The action of $G_{\p_\infty}$ on $W$ is via the natural map $\pi: G_{\p_\infty} \to \Delta$. It is easy to see that the image of the map $\pi$ is $\Delta_\wp$, the decomposition group of $\wp/p$. In this manner, we  view $W$ as an $G_{\p_\infty}$ representation.

\begin{lemma}\label{lemma2.1}

Let $W$ be a representation of $\Delta$. For any continuous character $\varphi: \Gamma:= \Gal(\Q_\infty/\Q) \to \OO^\times$, let $W(\varphi):= W \otimes \FF(\varphi)$ be the representation with diagonal action of $G_{\p_\infty}$. We have

\begin{enumerate}
\item $H^0(\Delta_\wp,W)=0$ iff $H^0(G_{\p_\infty}, W(\varphi))=0$. 
\item $H^0(\Delta_\wp, V)=0$ implies $H^0(G_{\p_\infty}, V_\wp^-)=0$ ($(\rho,V,V_{\wp}^+)$ satisfies {\bf HYP 1 \& 2}).
\end{enumerate}
In particular, $H^0(\Delta_\wp, U ) = 0$ implies $H^0(G_{\p_\infty}, U(\kappa)) = H^0(G_{\p_\infty}, V^* )=0$ and  $H^0(G_{\p_\infty}, U_\wp^-(\kappa)) = H^0(G_{\p_\infty}, (V^*)_\wp^- )=0$.
\end{lemma}
\begin{proof}
Note that as $G_{\p_\infty}$ acts on $W$ via $\Delta_\wp$, it follows that $H^0(\Delta_\wp, W)=0$ if and only if $H^0(G_{\p_\infty},W)=0$.
As $G_{\p_\infty}$ acts trivially on $\FF(\varphi)$, we have
 \begin{equation*}
 H^0(G_{\p_\infty}, W(\varphi)) = (W \otimes \FF(\varphi))^{G_{\p_\infty}} = W^{G_{\p_\infty}} \otimes \FF(\varphi).
 \end{equation*}
Hence $H^0(G_{\p_\infty},W)=0$ iff $H^0(G_{\p_\infty}, W(\varphi))=0$.\\
Next, we shall prove $(2)$. From the short exact sequence $0 \to V_\wp^+ \to V \to  V_\wp^- \to 0$, we get the exact sequence  $ H^0(\Delta_\wp, V) \to H^0(\Delta_\wp,  V_\wp^-) \to  H^1(\Delta_\wp, V_\wp^+)$. By  \cite[Proposition 1.6.2]{nsw},  $H^1(\Delta_\wp, V_\wp^+)=0$, hence $H^0(\Delta_\wp, V)=0$ implies $H^0(\Delta_\wp,  V_\wp^-)=0$. A similar argument shows that $H^0(\Delta_\wp, U)=0$ implies $H^0(G_{\p_\infty}, U_\wp^-)=0$
\end{proof}

\begin{rem}
We conclude that $H^0(\Delta_\wp, W)=0$ implies that $A_W(\varphi)^{G_{\p_\infty}}$ is finite for any $\varphi$. Now, let $G_n= G_{\p_n}$, note that $H^0(G_{\p_\infty},W(\varphi))=0$ implies $H^0(G_n, W(\varphi))=0$. As a consequence, we conclude that $H^0(\Delta_\wp,W)=0$ implies that $(A_W(\varphi))^{G_n}$ is finite for any $n$ and $\varphi$.
\end{rem}

The following theorem relates  the Selmer groups $Sel^{g}_{\Q_n}(A,A_\wp^+,\varphi)$ and $Sel^{GV}_{\Q_n}(A,A_\wp^+,\varphi)$.

\begin{theorem}\label{2.1}
Let $(\rho,V,V_{\wp}^+)$ be a triple which satisfies {\bf HYP 1 \& 2}. Assume that  $ H^0(\Delta_{\wp}, V)=0$ and $H^0(\Delta_{\wp}, U ) = 0$. Then for any continuous group homomorphism $\varphi: \Gamma \to \OO^\times$, the map $Sel^{g}_{\Q_n}(A,A_\wp^+,\varphi) \to Sel^{GV}_{\Q_n}(A,A_\wp^+,\varphi)$ is injective, and the cokernel of the map is finite for all $n$ with their cardinality uniformly bounded independent of $n$. In particular,  $Sel^{g}_{\Q_\infty}(A,A_\wp^+,\varphi)$ is a finite index subgroup of $Sel^{GV}_{\Q_\infty}(A,A_\wp^+,\varphi)$.
\end{theorem}

\begin{proof}
We want to show that the cardinality of $H^1_{GV}(G_{v_n}, A(\varphi))/H^1_{g}(G_{v_n}, A(\varphi))$ is finite and independent of $n$. 
First, we will consider the case where $ v_n \nmid p$ for all $n \geq 0$. In this situation, we follow  the proof as in \cite[Lemma 2.8]{oc}. Consider the following commutative diagram: 
\xymatrix{ 
& & H^1_g(G_{v_n}, V(\varphi)) \ar[r] \ar[d] & H^1_{GV}(G_{v_n}, A(\varphi))_{\mathrm{div}} \ar[d] \\
 0 \ar[r] & \frac{H^1(G_{v_n}, T(\varphi))}{H^1(G_{v_n}, T(\varphi))_{\mathrm{tor}}} \ar[r] \ar[d] & H^1(G_{v_n}, V(\varphi))  \ar[r] \ar[d] & H^1(G_{v_n}, A(\varphi))_{\mathrm{div}} \ar[r] \ar[d] & 0 \\ 
                   0 \ar[r] & \frac{H^1(I_{v_n}, T(\varphi))}{H^1(I_{v_n}, T(\varphi))_{\mathrm{tor}}} \ar[r] &  H^1(I_{v_n}, V(\varphi)) \ar [r] &  H^1(I_{v_n}, A(\varphi))_{\mathrm{div}} \ar[r] & 0. } 

Since $v_n \nmid p$, the group $H^1(I_{v_n}, T(\varphi))/H^1(I_{v_n}, T(\varphi))_{\tor}$ is a finitely generated $\Z_p$-module. 
Therefore, the image of $H^1_g(G_{v_n}, V(\varphi))$ coincides with the maximal divisible part $H^1_{GV}(G_{v_n}, A(\varphi))_{\mathrm{div}}$ of $H^1_{GV}(G_{v_n}, A(\varphi))$. 
To prove our result, it suffices to show that the cotorsion part of $H^1_{GV}(G_{v_n}, A(\varphi))= H^1(G_{v_n}/I_{v_n}, A(\varphi)^{I_{v_n}})$ is finite and independent on $n$.
Now consider the $I_{v_n}$-invariant subgroup $(A(\varphi))^{I_{v_n}}$ of $A(\varphi)$. Observe that $(A(\varphi))^{I_{v_n}}$ is independent of the choice of $n$ because $I_{v_n}=I_{v_\infty}$ for all $n \geq 0$; we denote $(A(\varphi))^{I_{v_n}}$ by $A^{'}(\varphi)$. 
Now consider the exact sequence $0 \to A^{'}(\varphi)_{\mathrm{div}} \to A^{'}(\varphi) \to A^{'}(\varphi)_{\mathrm{fin}} \to 0$, where $A^{'}(\varphi)_{\mathrm{fin}}=A^{'}(\varphi)/A^{'}(\varphi)_{\mathrm{div}}$ denote the finite part of $A^{'}(\varphi)$. 
This induces a map $H^1(G_{v_n}/I_{v_n}, A^{'}(\varphi)_{\mathrm{div}}) \to H^1(G_{v_n}/I_{v_n}, A^{'}(\varphi)) \to H^1(G_{v_n}/I_{v_n}, A^{'}(\varphi)_{\mathrm{fin}})$. 
The image of $H^1(G_{v_n}/I_{v_n}, A^{'}(\varphi)_{\mathrm{div}})$ under the above map is a co-free $\Z_p$-module. 
Hence the cotorsion part of $H^1_{GV}(G_{v_n}, A(\varphi))= H^1(G_{v_n}/I_{v_n}, A(\varphi)^{I_{v_n}})$ is bounded by $H^1(G_{v_n}/I_{v_n}, A^{'}(\varphi)_{\mathrm{fin}})$. Since  $G_{v_n}/I_{v_n}$ is a procyclic group, the order of $H^1(G_{v_n}/I_{v_n}, A^{'}(\varphi)_{\mathrm{fin}})$ is bounded by the order $A^{'}(\varphi)_{\mathrm{fin}}$. Therefore, the order of the cotorsion part of $H^1_{GV}(G_{v_n}, A(\varphi))$ and hence $H^1_{GV}(G_{v_n}, A(\varphi))/H^1_{g}(G_{v_n}, A(\varphi))$ is bounded by the order $A^{'}(\varphi)_{\mathrm{fin}}$. This proves the case where $v_n \nmid p$.
 
Next, consider the case where  $ v_n \mid p$, that is, $v_n = \p_n$. Consider the following commutative diagram as in \cite[Proposition 4.2, Page 87]{oc}
\begin{center}
\begin{equation}\label{3.a}
\begin{tikzpicture}[baseline= (a).base]
\node[scale=.9] (a) at (0,0){
\begin{tikzcd}[ arrows={-stealth}]
&  0 \dar & 0 \dar \\
& H^1_g(G_{\p_n}, V(\varphi)) \rar["{p_n}"] \dar & H^1_{GV}(G_{\p_n}, A(\varphi)) \dar \\
\frac{H^1(G_{\p_n}, T(\varphi))}{H^1(G_{\p_n}, T(\varphi))_{\mathrm{tor}}} \arrow[hookrightarrow]{r} \dar["{\alpha_n}"] &  H^1(G_{\p_n}, V(\varphi)) \rar \dar["{\beta_n}"] &  H^1(G_{\p_n}, A(\varphi)) \rar \dar["{\gamma_n}"]  & H^2(G_{\p_n}, T(\varphi))_{\mathrm{tor}} \\
{\Big( \frac{H^1(I_{\p_n}, T_\wp^-(\varphi))}{H^1(I_{\p_n}, T_\wp^-(\varphi))_{\mathrm{tor}}} \Big)}^{G_{\p_n}/I_{\p_n}} \arrow[hookrightarrow]{r} & H^1(I_{\p_n}, V_\wp^-(\varphi))^{G_{\p_n}/I_{\p_n}} \rar &  H^1(I_{\p_n},A_\wp^-(\varphi))^{G_{\p_n}/I_{\p_n}}.  & 
\end{tikzcd} };
\end{tikzpicture}
\end{equation}
 \end{center} 

The map $\beta_n$ is decomposed as $H^1(G_{\p_n}, V(\varphi)) \xrightarrow {\beta^{'}_{n}} H^1(G_{\p_n}, V_\wp^-(\varphi)) \xrightarrow {\beta^{''}_{n}} H^1(I_{\p_n}, V_\wp^-(\varphi))^{G_{\p_n}/I_{\p_n}}$.

The cokernel of $\beta^{'}_{n}$ is lies inside $H^2(G_{\p_n}, V_\wp^+(\varphi))$. By  Tate's duality, dim $H^2(G_{\p_n}, V_\wp^+(\varphi))=$ dim $H^0(G_{\p_n}, (V_\wp^+(\varphi))^*)=$ dim $H^0(G_{\p_n}, (V^*)_\wp^-(\varphi^{-1})).$ By our assumption $H^0(\Delta_\wp, U)=0$ and Lemma \ref{lemma2.1}, we have dim $H^0(G_{\p_n}, (V^*)_\wp^-(\varphi^{-1}))=0$. As a consequence the cokernel of $\beta^{'}_{n}$ is zero. Since the group $G_{\p_n}/I_{\p_n}$ has cohomological dimension one, the map $\beta^{''}_{n}$ is also surjective. Therefore, $\beta_n$ is surjective. 

Applying snake lemma to the commutative diagram in equation \eqref{3.a}, we get 
\begin{equation*}
 0 \to \mathrm{coker}(\alpha_{n}) \to \mathrm{coker}(p_n) \to H^2(G_{\p_n}, T(\varphi))_{\mathrm{tor}} .
\end{equation*}
By our assumption, $H^0(\Delta_\wp, U)=0$ and Lemma \ref{lemma2.1}, we obtain $H^0(G_{\p_\infty}, V^*(\varphi))=0$. Therefore, $H^0(G_{\p_n}, A^*(\varphi))$ is finite and independent on $n$, and hence by local Tate duality, $H^2(G_{\p_n}, T(\varphi))_{\mathrm{tor}}$ is also finite and independent on $n$. Thus to bound the size of the group $H^1_{GV}(G_{\p_n}, A(\varphi))/H^1_{g}(G_{\p_n}, A(\varphi)$, we  only need to bound the size of $\mathrm{coker}(\alpha_{n})$. Now consider the following commutative diagram:
\begin{center}
$
\xymatrix{ 
H^1(G_{\p_n}, T(\varphi))_{\mathrm{tor}} \ar@{^{(}->}[r] \ar[d] & H^1(G_{\p_n}, T(\varphi))  \ar[r] \ar[d]_{\psi_n} & \frac{H^1(G_{\p_n}, T(\varphi))}{H^1(G_{\p_n}, T(\varphi))_{\mathrm{tor}}} \ar[r] \ar[d]_{\alpha_n} & 0 \\ 
H^1(I_{\p_n}, T_\wp^-(\varphi))_{\mathrm{tor}}^{G_{\p_n}/I_{\p_n}} \ar@{^{(}->}[r] &  H^1(I_{\p_n}, T_\wp^-(\varphi))^{G_{\p_n}/I_{\p_n}} \ar [r]^{e_{n}} &  {\Big( \frac{H^1(I_{\p_n}, T_\wp^-(\varphi))}{H^1(I_{\p_n}, T_\wp^-(\varphi))_{\mathrm{tor}}} \Big)}^{G_{\p_n}/I_{\p_n}}. }$
\end{center}
The coker$(e_n)$ is a submodule of $H^1(G_{\p_n}/I_{\p_n}, H^1(I_{\p_n}, T_\wp^-(\varphi))_{\mathrm{tor}})$. Therefore by applying snake lemma we obtain $\mathrm{coker}(\psi_{n}) \to \mathrm{coker}(\alpha_n) \to H^1(G_{\p_n}/I_{\p_n}, H^1(I_{\p_n}, T_\wp^-(\varphi))_{\mathrm{tor}})$.

We now consider the following commutative diagram:
\begin{center}
\xymatrix{ 
0 \ar[r] & H^0(I_{\p_n}, A_\wp^-(\varphi))_{\mathrm{div}} \ar[r] \ar[d]_{1-g_{\p_n}} & H^0(I_{\p_n}, A_\wp^-(\varphi))  \ar[r] \ar[d]_{1-g_{\p_n}} & H^1(I_{\p_n}, T_\wp^-(\varphi))_{\mathrm{tor}} \ar[r] \ar[d]_{1-g_{\p_n}} & 0 \\ 
0 \ar[r] & H^0(I_{\p_n}, A_\wp^-(\varphi))_{\mathrm{div}} \ar[r] &  H^0(I_{\p_n}, A_\wp^-(\varphi)) \ar[r] &  H^1(I_{\p_n}, T_\wp^-(\varphi))_{\mathrm{tor}} \ar[r] & 0, } 
\end{center}
where $g_{\p_n}$ is a topological generator of the pro-cyclic group $G_{\p_n}/I_{\p_n}$. The cokernel of the right vertical map is  $H^1(G_{\p_n}/I_{\p_n}, H^1(I_{\p_n}, T_\wp^-(\varphi))_{\mathrm{tor}})$. The kernel of the middle vertical map is $H^0(G_{\p_n}, A_\wp^-(\varphi))$. The group $H^0(G_{\p_n}, A_\wp^-(\varphi))$  is finite and independent on $n$ by our assumption $ H^0(\Delta_\wp, V)=0$ and Lemma \ref{lemma2.1}.  Hence cokernel of the middle vertical map is also is finite and independent on $n$. From snake lemma, we obtain that $H^1(G_{\p_n}/I_{\p_n}, H^1(I_{\p_n}, T_\wp^-(\varphi))_{\mathrm{tor}})$ is finite and independent of $n$. Thus to show that coker$(\alpha_{n})$ is finite and independent on $n$, it is enough to show it for coker$(\psi_{n})$. The map $\psi_{n}$ is decomposed as $H^1(G_{\p_n}, T(\varphi)) \xrightarrow {\psi^{'}_{n}} H^1(G_{\p_n}, T_\wp^-(\varphi)) \xrightarrow {\psi^{''}_{n}} H^1(I_{\p_n}, T_\wp^-(\varphi))^{G_{\p_n}/I_{\p_n}}$.
The map ${\psi^{''}_{n}}$ is surjective as $G_{\p_n}/I_{\p_n}$ has cohomological dimension $1$. 
To prove our claim, we just need to show that  coker$(\psi^{'}_{n})$ is bounded and independent on $n$. The cokernel of $\psi^{'}_{n}$ is a submodule of $H^2(G_{\p_n}, T_\wp^+(\varphi))$. The group $H^2(G_{\p_n}, T_\wp^+(\varphi))$ is pontryagin dual of $H^0(G_{\p_n}, A_{\wp}^+(\varphi)^{*})$, which is finite and independent on $n$ by our assumption $ H^0(\Delta_\wp, V)=0$ and Lemma \ref{lemma2.1}. 

\end{proof}

Comparison between the Selmer groups $Sel^{g}_{\Q_n}(A^*,{(A^*)}_\wp^+,\varphi)$ and $Sel^{GV}_{\Q_n}(A^*,{(A^*)}_\wp^+,\varphi)$ is similar. Recall that from $(\rho,V,V_{\wp}^+)$ we get an Artin Representation $(\sigma,U,U_{\wp}^+)$, where $U=$Hom$_{\FF}(V, \FF(\tau))$ and we have $V^* \cong U \otimes_\FF \FF(\kappa) := U(\kappa)$. Recall that using the pairing between $U$ and $V$, we defined a $\Delta_\wp$ filtration on $U$ (hence on $V^*(\varphi)$) 
$$0 \to U_\wp^+ \to U \to U_\wp^-:= U/U_\wp^+ \to 0,$$
where $U_\wp^+$ is the orthogonal compliment of $V_\wp^+$ under the pairing between $U$ and $V$.

\begin{corollary} \label{2.3}
Let $(\rho,V,V_{\wp}^+)$ be a triple which satisfies {\bf HYP 1 \& 2}. Assume that $H^0(\Delta_\wp, V) = 0$ and $H^0(\Delta_\wp, U)= 0$. Let $\varphi: \Gamma \to \OO^\times$ be a continuous group homomorphism. 
Then the map $Sel^{g}_{\Q_n}(A^*,{(A^*)}_\wp^+,\varphi) \to Sel^{GV}_{\Q_n}(A^*,{(A^*)}_\wp^+,\varphi)$ is injective, and the cokernel of the map is finite for all $n$ with their cardinality uniformly bounded independent of $n$. In particular,  $Sel^{g}_{\Q_\infty}(A^*,{(A^*)}_\wp^+,\varphi)$ is a finite index subgroup of $Sel^{GV}_{Q_\infty}(A^*,{(A^*)}_\wp^+,\varphi)$.
\end{corollary}
\begin{proof}
Recall that $V^*(\varphi) \cong U (\varphi\kappa)$. One easily sees that $(\rho,V,V_{\wp}^+)$ satisfies hypothesis {\bf HYP 1 \& 2} implies that $(\sigma,U,U_{\wp}^+)$ also satisfies hypothesis {\bf HYP 1 \& 2}. Since $(V^*(\varphi))^+_\wp = U_\wp^+(\varphi \kappa)$ and $(V^*(\varphi))^-_\wp = U_\wp^-(\varphi \kappa)$, we get our desired result by replacing $V$ by $U$ and $\varphi$ by $\varphi \kappa$ in Theorem \ref{2.1}.

\end{proof}

\section{Control Theorems}
In this section, we prove a `control theorem' for the Selmer groups $Sel^{GV}_{L}(B,B_\wp^+,\varphi)$ and $Sel^{g}_{L}(B,B_\wp^+,\varphi)$ defined in section \ref{section1}, where $L \in \{ \Q_n , \Q_{\infty} \}$ and $B \in \{ A , A^* \}$. Mazur \cite{ma} first proved `control theorem' for Selmer groups associated to Abelian varieties. Since then many such `control theorems' have arisen naturally in the study of Selmer groups. For example, control theorem for Bloch-Kato Selmer group associated to $p$-adic representation \cite{oc},  for $p$-stabilized newform  \cite{jp}, for nearly ordinary Hilbert modular over the cyclotomic $\Z_p$ extension of a totally real number field \cite{jm} can be found in the literature. For Selmer group associated to Artin representations, Greenberg-Vatsal \cite{gv} proved a similar control theorem when $d^+ =1$.

\begin{theorem} \label{3.1}
Let $(\rho,V,V_{\wp}^+)$ be a triple which satisfies {\bf HYP 1 \& 2}. Assume that $H^0(\Delta_\wp, V) = 0$. Let $\varphi: \Gamma \to \OO^*$ be a continuous group homomorphism.  Then the map $Sel^{GV}_{\Q_n}(A,A_\wp^+,\varphi) \to Sel^{GV}_{\Q_\infty}(A,A_\wp^+,\varphi)^{\Gamma_n}$ is injective, and the cokernel of the map is finite for all $n$ with their cardinality uniformly bounded independent on $n$.
\end{theorem}

\begin{proof}
Consider this following commutative diagram:

\begin{tikzpicture}[baseline= (a).base]
\node[scale=.8] (a) at (0,0){
\begin{tikzcd}[ arrows={-stealth}]
0 \rar & Sel^{GV}_{\Q_n}(A,A_\wp^+,\varphi) \rar \dar["{\alpha_n}"] &  H^1(\Q_\Sigma/\Q_{n}, A(\varphi)) \rar \dar["{\beta_n}"] &  \underset { v_n \mid \Sigma, v_n \nmid p\infty}{\oplus}  H^1(I_{v_n},A(\varphi)) \underset { v_n= \p_n}{\oplus} H^1(I_{\p_n}, A_\wp^-(\varphi))  \dar["{\gamma_n = \oplus \gamma_{v_n}}"] \\
0 \rar &Sel^{GV}_{\Q_\infty}(A,A_\wp^+,\varphi)^{\Gamma_n} \rar & H^1(\Q_\Sigma/\Q_{\infty}, A(\varphi))^{\Gamma_n} \rar &  \underset { v_\infty \mid \Sigma, v_\infty \nmid p\infty}{\oplus}  H^1(I_{v_\infty},A(\varphi))^{\Gamma_n}  \underset { v_\infty = \p_ \infty}{\oplus} H^1(I_{\p_\infty}, A_\wp^-(\varphi))^{\Gamma_n} 
\end{tikzcd} };
\end{tikzpicture} \\
By inflation-restriction exact sequence, the kernel of $\beta_n$ is $H^1(\Gal(\Q_{\infty}/\Q_{n}), A(\varphi)^{G_{\Q_{\infty}}})$. 
Since $p \nmid |\Delta|$, we have $ \Delta \cong \Gal(K_n/ \Q_n)$, where $K_n=K \Q_n$ is the $n$-th layer of the cyclotomic $\Z_p$-extension $K_{\infty}$ of $K$. Now for any $n$, the action of $G_{\Q_{n}}$ factors through $\Gal(K_n/\Q_n)$ and hence $ A(\varphi)^{G_{\Q_{n}}} =  A(\varphi)^{\Gal(K_n/\Q_n)} =  A(\varphi)^{\Delta}= A(\varphi)^{G_\Q}$. As a consequence, we obtain $A(\varphi)^{G_{\Q_{\infty}}} = A(\varphi)^{G_\Q} = A(\varphi)^{\Delta}$. Consider the exact sequence
\begin{equation*}
 H^0( \Delta, V(\varphi) ) \to H^0( \Delta, A(\varphi) ) \to  H^1( \Delta, T(\varphi) ).
\end{equation*}

The group $H^0( \Delta, V(\varphi))=0$ because $\rho$ is a non trivial irreducible representation and $H^1( \Delta, T(\varphi) )=0$ by ( \cite[Proposition 1.6.2]{nsw}), therefore  $A(\varphi)^{G_{\Q_{\infty}}}=A(\varphi)^{G_{\Q}}=0$. Hence $\beta_{n}$ is injective. As a consequence, we see that $\alpha_n$ is also injective.

We want to show that coker($\alpha_{n}$) is finite and independent on $n$. First, observe that the cohomological dimension of $\Gal(\Q_{\infty}/\Q_{n})$ is one; therefore, $\beta_{n}$ is surjective. Hence it is sufficient to bound the kernel of $\gamma_{n}.$

First, we consider the case $ v_n\nmid p$. In this situation, by inflation restriction, we have $\Ker(\gamma_{v_n}) \cong  H^1(I_{v_\infty}/I_{v_n}, A(\varphi)^{I_{v_\infty}})$. Since $v_n \nmid p$, the extension $\Q_{\infty}/\Q_n$ is unramified. Therefore, $I_{v_\infty/v_n}=0$. Hence the map $\gamma_{v_n}$ is injective for $v_n \nmid p$.

Finally, consider the situation $\gamma_{v_n}$, where $v_n \mid p$, that is, $v_n = \p_n$. Recall that $V$ has a $\Delta_\wp$-invariant subspace $V_\wp^+$ of dimension $d^+$, which gives us the following spaces $V_\wp^-(\varphi) := V(\varphi)/V_\wp^+(\varphi)$ , $T_\wp^-(\varphi) := T(\varphi)/T_\wp^+(\varphi)$ and $A_\wp^-(\varphi) := A(\varphi)/A_\wp^+(\varphi)$. Consider the following commutative diagram 
\begin{center}
\begin{tikzcd}
 H^1(G_{\p_n}, A_\wp^-(\varphi)) \rar["a_n"] \dar["{b_n}"] &  H^1(G_{\p_\infty}, A_\wp^-(\varphi)) \dar["\theta"]  \\
H^1(I_{\p_n},A_\wp^-(\varphi))^{G_{\p_n}/I_{\p_n}} \rar["\gamma_{\p_n}"] &  H^1(I_{\p_\infty},A_\wp^-(\varphi))^{G_{\p_\infty}/I_{\p_\infty}}.
\end{tikzcd}
\end{center}
The kernel of the map $a_n$ is ker $(a_n)=H^1(G_{\p_n}/G_{\p_\infty}, A_\wp^-(\varphi)^{G_{\p_\infty}})$. From the exact sequence $0 \to T_\wp^-(\varphi) \to V_\wp^-(\varphi) \to  A_\wp^-(\varphi) \to 0$, we obtain an exact sequence 
\begin{equation*}
0 =H^0(\Delta_\wp, V_\wp^-(\varphi)) \to H^0(\Delta_\wp,  A_\wp^-(\varphi)) \to  H^1(\Delta_\wp, T_\wp^-(\varphi)).
\end{equation*} 

The fact $0 =H^0(\Delta_\wp, V_\wp^-(\varphi))$ follows from our assumption $H^0(\Delta_\wp, V)=0$ and Lemma \ref{lemma2.1}. The order of $\Delta_\wp$ is co-prime to $p$, therefore $H^1(\Delta_\wp, T_\wp^-(\varphi))=0$ (See \cite[Proposition 1.6.2]{nsw}). Hence, $H^0(\Delta_\wp,  A_\wp^-(\varphi))=0$. The $G_{\p_\infty}$ action factors through the group $\Delta_\wp$, it follows that $H^0(G_{\p_\infty}, A_\wp^-(\varphi))=0$. Therefore $a_n$ is injective. 
Since the cohomological dimension of $G_{\p_n}/I_{\p_n}$ is one, the map $b_n$ is surjective. By inflation-restriction map ker$(\theta)=H^1(G_{\p_\infty}/I_{\p_\infty}, {A_\wp^-(\varphi)}^{I_{\p_\infty}}).$ Since $G_{\p_\infty}/I_{\p_\infty}$ is a procyclic group, corank $H^1(G_{\p_\infty}/I_{\p_\infty}, {A_\wp^-(\varphi)}^{I_{\p_\infty}})$ = corank $H^0(G_{\p_\infty}/I_{\p_\infty}, {A_\wp^-(\varphi)}^{I_{\p_\infty}})$ = corank $H^0(G_{\p_\infty} A_\wp^-(\varphi)) =0$. 
Therefore  Ker$(\theta)=H^1(G_{\p_\infty}/I_{\p_\infty}, {(A_\wp^-(\varphi))}^{I_{\p_\infty}})$ is finite. 
So by surjectivity of $b_n$ and commutativity of the diagram, we can say that kernel of $\gamma_{\p_n}$ is bounded and independent on $n$.
\end{proof}
\begin{corollary}\label{3.2}
Let $(\rho,V,V_{\wp}^+)$ be a triple which satisfies {\bf HYP 1 \& 2}. Assume that $H^0(\Delta_\wp, U) = 0$. Let $\varphi: \Gamma \to \OO^*$ be a continuous group homomorphism. Then the map $Sel^{GV}_{\Q_n}(A^*,{(A^*)}_\wp^+,\varphi) \to Sel^{GV}_{\Q_\infty}(A^*,{(A^*)}_\wp^+,\varphi)^{\Gamma_n}$ is injective, and the cokernel of the map is finite for all $n$ with their cardinality uniformly bounded independent on $n$.
\end{corollary}

\begin{proof}
The proof is similar to the proof of Theorem \ref{3.1}.

\end{proof}

Next, we prove a control theorem for $Sel^{g}_{L}(B,B_\wp^+,\varphi)$, where $L \in \{ \Q_n , \Q_{\infty} \}$ and $B \in \{ A , A^* \}$. To do this, we need the following result regarding the duality of the local Selmer condition at the prime $p$.

\begin{lemma}\label{3.3}
Assume that $H^0(\Delta_\wp ,V ) = H^0(\Delta_\wp, U)=0$. Let $\varphi: \Gamma \to \OO^*$ be any continuous group homomorphism. For any $n \ge 0$, let $v_n \mid \Sigma$ be a finite place in $\Q_n$.  Then under the duality between $H^1(G_{v_n}, V(\varphi))$ and   $H^1(G_{v_n}, V(\varphi)^*)$, the orthogonal complement of $H^1_g(G_{v_n}, V(\varphi))$ is $H^1_g(G_{v_n}, V(\varphi)^*)$. \\
As a consequence, under the perfect pairing between $H^1(G_{v_n}, T(\varphi)^*)$ and   $H^1(G_{v_n}, A(\varphi))$, the orthogonal complement of $H^1_g(G_{v_n}, T(\varphi)^*)$ is $H^1_g(G_{v_n}, A(\varphi))$, where $H^1_g(G_{v_n}, T(\varphi)^*)$ is defined as 
\begin{equation*}
H^1_g(G_{\p_n}, T(\varphi)^*) = \Ker \Big( H^1(G_{\p_n}, T(\varphi)^*) \xrightarrow H^1(G_{\p_n}, V(\varphi)^*) \to H^1(I_{\p_n}, {(V(\varphi)^*)}_\wp^-) \Big ).
\end{equation*}
\end{lemma}

\begin{proof} 
First, we consider the case where $v_n \nmid p .$ Recall that $H^1_g(G_{v_n}, V(\varphi)):= \Ker (H^1(G_{v_n}, V(\varphi)) \to H^1(I_{v_n}, V(\varphi)) ) \cong 
H^1(G_{v_n}/I_{v_n},V(\varphi)^{I_{v_n}})$ 
by the inflation-restriction exact sequence. Hence the image of $H^1_g(G_{v_n}, V(\varphi)) \otimes H^1_g(G_{v_n}, {(V(\varphi))}^*)$ in $H^2(G_{v_n},\FF(\tau\kappa))$ is $0$, since it lies in $H^2(G_{v_n}/I_{v_n},\FF(\tau\kappa))$ and the cohomological dimension of $G_{v_n}/I_{v_n}$ is one.
To conclude that the orthogonal compliment of $H^1_g(G_{v_n}, V(\varphi))$ is $H^1_g(G_{v_n}, V(\varphi)^*)$, we only have to show that $\dim H^1_g(G_{v_n}, V(\varphi))$ + $\dim H^1_g(G_{v_n}, V(\varphi)^*)$ = $\dim H^1(G_{v_n}, V(\varphi))$.
Note that,
\begin{equation}\label{eq1} 
 \dim H^1_g(G_{v_n}, V(\varphi)) = \dim H^1(G_{v_n}/I_{v_n},V(\varphi)^{I_{v_n}}) =H^0(G_{v_n},V(\varphi)).
\end{equation} 

From equation \eqref{eq1} we obtain, $\dim H^1_g(G_{v_n}, V(\varphi)) + \dim H^1_g(G_{v_n}, V(\varphi)^*) = \dim H^0(G_{v_n}, V(\varphi))+ \dim H^0(G_{v_n}, V(\varphi)^*)$. This is same as $\dim H^1(G_{v_n}, V(\varphi))$, by Euler-Poincar$\acute{e}$ characteristic formula. This completes the proof of the lemma in this case.

Next, consider the case  where $v_n = \p_n \mid p$. Recall that in this case $H^1_g(G_{\p_n}, V(\varphi)):=  \Ker(H^1(G_{\p_n}, V(\varphi)) \xrightarrow{\alpha} H^1(G_{\p_n}, V(\varphi)_\wp^-) \xrightarrow{\beta} H^1(I_{\p_n}, V(\varphi)_\wp^-))$. By the inflation-restriction sequence we have $\Ker(\beta) = H^1(G_{\p_n}/I_{\p_n},  {(V(\varphi)_\wp^-)}^{I_{\p_n}}) $. Now $\dim H^1(G_{\p_n}/I_{\p_n}, {(V(\varphi)_\wp^-)}^{I_{\p_n}})  = \dim H^0(G_{\p_n},V(\varphi)_\wp^-)$, which is $0$ by our assumption that $H^0(\Delta_\wp,V)=0$ and Lemma \ref{lemma2.1}. Hence $\beta$ is injective and as a consequence we obtain that 
\begin{equation}\label{eq1'}
H^1_g(G_{\p_n}, V(\varphi))=  \Ker(H^1(G_{\p_n}, V(\varphi)) \xrightarrow{\alpha} H^1(G_{\p_n}, V(\varphi)_\wp^-)) = \Ker(\alpha).
\end{equation}
From our assumption $H^0(\Delta_\wp, U)=0$ and Lemma \ref{lemma2.1}, we get that $H^0(G_{\p_\infty},U_\wp^-(\kappa))=0$. Therefore $H^2(G_{\p_n}, V(\varphi)_\wp^+) = H^0(G_{\p_n}, {(V(\varphi)_\wp^+)}^*)=0$ because of Lemma \ref{lemma2.1} and the fact that $({V(\varphi)_\wp^+})^* \cong {(V(\varphi)^*})_\wp^-$. We obtain a long exact sequence of cohomology
\begin{equation}\label{eq2} 
 0= H^0(G_{\p_n},V(\varphi)_\wp^-) \to H^1(G_{\p_n}, V(\varphi)_\wp^+) \to  H^1(G_{\p_n}, V(\varphi)) \xrightarrow {\alpha} H^1(G_{\p_n}, V(\varphi)_\wp^-)\to H^2(G_{\p_n}, V(\varphi)_\wp^+)=0.
\end{equation}
 
Now using equations \eqref{eq1'} and \eqref{eq2}, we obtain 
\begin{equation}\label{eq2'}
H^1_g(G_{\p_n}, V(\varphi)) = \Ker(\alpha) = H^1(G_{\p_n}, V(\varphi)_\wp^+).
\end{equation}
Similarly, we obtain
\begin{equation}\label{eq2''}
%H^1_g(G_{\p_n}, {V(\varphi)}^*) = H^1(G_{\p_n}, {(V(\varphi)^*)}_\wp^+) = H^1(G_{\p_n},({V(\varphi)}_\wp^-)^*).
 \begin{split}
	H^1_g(G_{\p_n}, {V(\varphi)}^*) & = \Ker \Big ( H^1(G_{\p_n}, V(\varphi)^*) \to H^1(G_{\p_n}, (V(\varphi)^*)_\wp^-) \to H^1(I_{\p_n}, (V(\varphi)^*)_\wp^-)  \Big ), \\
	& = \Ker \Big ( H^1(G_{\p_n}, V(\varphi)^*) \to H^1(G_{\p_n}, (V(\varphi)^*)_\wp^-) \Big ) = H^1(G_{\p_n},({V(\varphi)}_\wp^-)^*),
	\end{split}	
\end{equation}
for the last equality, we have used $({V(\varphi)}_\wp^-)^* \cong (V(\varphi)^*)_\wp^+$.\\

First, we will show that $H^1(G_{\p_n}, V(\varphi)_\wp^+)$ and $H^1(G_{\p_n}, ({V(\varphi)}_\wp^-)^*)$ are orthogonal complement to each other under the duality between $H^1(G_{\p_n}, V(\varphi))$ and   $H^1(G_{\p_n}, V(\varphi)^*)$. \\

Under the pairing between $V(\varphi)$ and $V(\varphi)^*$, orthogonal complement of $V(\varphi)_\wp^+$ is $({V(\varphi)}_\wp^-)^* \cong (V(\varphi)^*)_\wp^+$. So the image of $H^1(G_{\p_n}, V(\varphi)_\wp^+) \otimes H^1(G_{\p_n}, ({V(\varphi)}_\wp^-)^*)$ is $0$. To complete the proof, we need to show
\begin{equation}\label{eq2a}
\dim H^1(G_{\p_n}, V(\varphi)_\wp^+) + \dim H^1(G_{\p_n}, ({V(\varphi)}_\wp^-)^*) = \dim  H^1(G_{\p_n}, V(\varphi)).
\end{equation}

From equation \eqref{eq2}, we see that 
\begin{equation}\label{eq2b}
 \dim H^1(G_{\p_n}, V(\varphi)) = \dim  H^1(G_{\p_n}, V(\varphi)_\wp^+) + \dim H^1(G_{\p_n}, V(\varphi)_\wp^-).
\end{equation}
Equation \eqref{eq2a} follows from the equation \eqref{eq2b} and the fact $\dim H^1(G_{\p_n}, ({V(\varphi)}_\wp^-)^*) = \dim H^1(G_{\p_n}, {V(\varphi)}_\wp^-)$.

\end{proof}

Now we proceed to prove a control theorem which connects $Sel^{g}_{\Q_n}(A,A_\wp^+,\varphi)$ and $Sel^{g}_{\Q_\infty}(A,A_\wp^+,\varphi)$.

\begin{theorem} \label{3.4}
Let $(\rho,V,V_\wp^+)$ be a triple which satisfies {\bf HYP 1 \& 2}. Assume that  $H^0(\Delta_\wp,V) = H^0(\Delta_\wp, U)=0$. Let $\varphi: \Gamma \to \OO^*$ be a continuous group homomorphism.  Then the map $Sel^{g}_{\Q_n}(A,A_\wp^+,\varphi)$ \\ $ \to Sel^{g}_{\Q_\infty}(A,A_\wp^+,\varphi)^{\Gamma_{n}}$ is injective, and the cokernel of this map is finite for all $n$ with cardinality bounded independent of $n$.
\end{theorem}

\begin{proof}

Let us consider the following commutative diagram:
\begin{center}
\xymatrix { 
0 \ar[r] & Sel^{g}_{\Q_n}(A,A_\wp^+,\varphi)  \ar[r] \ar[d]_{\alpha_n} & H^1(\Q_\Sigma/\Q_{n}, A(\varphi))   \ar[r] \ar[d]_{\beta_{n}} &  \underset { v_n \mid \Sigma, v_n \nmid p\infty}{\oplus} \frac{ H^1(G_{v_n}, A(\varphi))}{H^1_g(G_{v_n},A(\varphi))} \ar[d]_{\gamma_n=\oplus \gamma_{v_n}} \underset { v_n = \p_n}{\oplus} \frac{H^1(G_{\p_n}, A(\varphi))}{H^1_g(G_{\p_n}, A(\varphi))} 
\\ 
0 \ar[r] & Sel^{g}_{\Q_\infty}(A,A_\wp^+,\varphi)^{\Gamma_n} \ar[r] &  H^1(\Q_\Sigma/\Q_{\infty}, A(\varphi))^{\Gamma_n} \ar [r] & \Bigg( \underset { v_\infty \mid \Sigma, v_\infty \nmid p\infty}{\oplus} \frac{ H^1(G_{v_\infty}, A(\varphi))}{H^1_g(G_{v_\infty},A(\varphi))} \underset { v_\infty = p_\infty}{\oplus} \frac{H^1(G_{\p_\infty}, A(\varphi))}{H^1_g(G_{\p_\infty}, A(\varphi))} \Bigg)^{\Gamma_{n}} } 
\end{center}

As in Theorem \ref{3.1}, we can show that $\alpha_n$ is injective, and to prove that coker($\alpha_{n}$) is finite and independent on $n$, it is sufficient to bound the kernel of $\gamma_{n}.$ 
 
First, we consider the case $v_n \nmid p$. In this case, definition of $H^1_g$ matches with the definition of $H^1_f$ as in \cite[page 73]{oc}. The fact that kernel of $\gamma_{v_n}$ is bounded and  independent on $n$, follows from \cite[Lemma 2.8]{oc}.\\
Finally, we need to bound the kernel of $\gamma_{v_n}$ when $v_n = \p_n$. To bound the kernel, it is enough to bound the cokernel of the dual map 
\begin{equation*}
\Bigg ( \frac{ H^1(G_{\p_\infty}, A(\varphi))}{H^1_g(G_{\p_\infty},A(\varphi))} \Bigg )^\vee \xrightarrow {\gamma_{\p_n}^\vee} \Bigg ( \frac{ H^1(G_{\p_n}, A(\varphi))}{H^1_g(G_{\p_n},A(\varphi))} \Bigg )^\vee.
\end{equation*}
By Lemma \ref{3.3}, under the perfect pairing between $H^1(G_{v_n}, T(\varphi)^*)$ and   $H^1(G_{v_n}, A(\varphi))$, the orthogonal complement of $H^1_g(G_{v_n}, T(\varphi)^*)$ is $H^1_g(G_{v_n}, A(\varphi))$. Therefore, $H^1_g(G_{v_n}, T(\varphi)^*)$ is the pontryagin dual of $H^1(G_{\p_n}, A(\varphi))/H^1_g(G_{\p_n},A(\varphi))$, hence it is equivalent to bound the cokernel of
\begin{equation*}
\theta_n: \varprojlim_m H^1_g(G_{\p_m}, T(\varphi)^*)  \xrightarrow {\theta_{n}}  H^1_g(G_{\p_n},T(\varphi)^*).
\end{equation*}
Recall that $H^1_g(G_{\p_n}, T(\varphi)^*)$ is defined as
\begin{equation}\label{eq4}
H^1_g(G_{\p_n}, T(\varphi)^*) = \Ker \Big( H^1(G_{\p_n}, T(\varphi)^*) \xrightarrow {i_n} H^1(G_{\p_n}, V(\varphi)^*) \to H^1(I_{\p_n}, {(V(\varphi)^*)}_\wp^-) \Big ).
\end{equation}
Hence from equations \eqref{eq2''} and \eqref{eq4}, we obtain
\begin{equation}\label{eq7}
H^1_g(G_{\p_n}, T(\varphi)^*) = \Ker \Big( H^1(G_{\p_n}, T(\varphi)^*) \xrightarrow {i_n} H^1(G_{\p_n}, V(\varphi)^*) \xrightarrow {b'_n} H^1(G_{\p_n}, (V(\varphi)^*)_\wp^-) \Big ).
\end{equation}

Now consider the commutative diagram:

\begin{tikzcd}[ arrows={-stealth}]
0=H^0(G_{\p_n},{(T(\varphi)^*)}_\wp^-)  \rar &  H^1(G_{\p_n}, {(T(\varphi)^*)}_\wp^+) \rar["{a_n}"] \dar 
& H^1(G_{\p_n}, T(\varphi)^*)\rar["{b_n}"] \dar["{i_n}"]\arrow[dr, dashrightarrow,xshift=-1ex] & H^1(G_{\p_n}, {(T(\varphi)^*)}_\wp^-) \dar["{i'_n}"] \\
0=H^0(G_{\p_n},{(V(\varphi)^*)}_\wp^-) \rar & H^1(G_{\p_n}, {(V(\varphi)^*)}_\wp^+) \rar["{a'_n}"]
  &  H^1(G_{\p_n}, V(\varphi)^*) \rar["{b'_n}"] & H^1(G_{\p_n}, {(V(\varphi)^*)}_\wp^-).
\end{tikzcd}

From the above commutative diagram and equation \eqref{eq7}, we have   
\begin{equation*}
H^1_g(G_{\p_n}, T(\varphi)^*)= \Ker(b'_n \circ i_n) = \Ker(i'_n \circ b_n).  
\end{equation*}
From snake lemma, we get $0 \to \Ker (b_n) = H^1(G_{\p_n}, {(T(\varphi)^*)}_\wp^+) \to \Ker(i'_n \circ b_n) \to \Ker(i'_n) = H^0(G_{\p_n}, {(A(\varphi)^*)}_\wp^-)=0 $, for the last equality we used the fact that  $H^0(G_{\p_n},{(V(\varphi)^*)}_\wp^-)=0$, and $p$ does not divide the order of the group $\Delta$. Therefore, $ H^1(G_{\p_n}, {(T(\varphi)^*)}_\wp^+) \cong H^1_g(G_{\p_n}, T(\varphi)^*).$
Hence it is sufficient to show that the cokernel of $\phi_{n} : \varprojlim_m H^1(G_{\p_m}, {(T(\varphi)^*)}_\wp^+) \to H^1(G_{\p_n}, {(T(\varphi)^*)}_\wp^+)$ bounded and independent on n. Now 
\begin{align}\label{eq9}
 \text{Coker}(\phi_n)& = \text{Coker} (\varprojlim_m H^1(G_{\p_m}, {(T(\varphi)^*)}_\wp^+) \to H^1(G_{\p_n}, {(T(\varphi)^*)}_\wp^+)) \nonumber &\\
 & \cong \text{Coker} (\varprojlim_m H^1(G_{\p_m}, (T(\varphi)_\wp^-)^*) \to H^1(G_{\p_n}, (T(\varphi)_\wp^-)^*). 
\end{align} 
Here for the last equality, we have used the fact that ${(T(\varphi)^*)}_\wp^+ \cong (T(\varphi)_\wp^-)^*$. Now taking the Pontryagin dual of the equation \eqref{eq9}, we obtain
\begin{align*} 
\text{Coker}(\phi_n)^\vee & = \text{Ker}  (H^1(G_{\p_n}, A(\varphi)_\wp^-) \to \varinjlim_m H^1(G_{\p_m}, A(\varphi)_\wp^-))\nonumber  & \\ 
 & =  \text{Ker} (H^1(G_{\p_n}, A(\varphi)_\wp^-) \to  H^1(G_{\p_\infty}, A(\varphi)_\wp^-)) \nonumber =  H^1(G_{\p_n}/G_{\p_\infty}, (A(\varphi)_\wp^-)^{G_{\p_\infty}}). &
\end{align*}
Since $(A(\varphi)_\wp^-)^{G_{\p_\infty}}=0$,  we obtain $\text{Coker}(\phi_n)^\vee = 0$.
\end{proof}

\begin{corollary}
Let $(\rho,V,V_\wp^+)$ be a triple which satisfies {\bf HYP 1 \& 2}. Assume that  $H^0(\Delta_\wp, U ) = H^0(\Delta_\wp, V)=0$. Let $\varphi: \Gamma \to \OO^*$ be a continuous group homomorphism. Then the map  $Sel^{g}_{\Q_n}(A^*,{(A^*)}_\wp^+,\varphi) \to Sel^{g}_{\Q_\infty}(A^*,{(A^*)}_\wp^+,\varphi)^{\Gamma_n}$ is injective, and the  cokernel of this map is finite for all $n$ with their cardinality bounded independent of $n$.
\end{corollary}

\begin{proof}
The proof is similar to the proof of Theorem \ref{3.4}.

\end{proof}

\section{Characteristic ideals}
In this section, we show that in the case of an Artin representation,  the characteristic ideals of the dual Selmer group do not depend on the choice of the lattice $T$, provided that the Artin representation satisfies a mild hypothesis. For elliptic curves, Mazur \cite{ma} showed that the characteristic ideals are different for a different choice of lattices, and variation of $\mu$-invariant under isogeny was studied in \cite{sc, pr1}. In the case of Artin representation over totally real fields, it was shown in ( \cite[Proposition 2.4]{gr}) that the characteristic ideals are independent on the choice of lattices. We will follow a similar idea given in ( \cite[Proposition 2.4]{gr}) and ( \cite[Proposition 4.1.1]{gr1}).

Let $T$ and $T_1$ be two $\Delta$-invariant $\OO$-lattices in $V$ and $A= V/T$ , $A_1= V/T_1$. %Then we can consider $X^{GV}_{\Q_\infty}(T,T{_\wp^+})$ and $X^{GV}_{\Q_\infty}(T_1,(T_1){_\wp^+})$ as a $\Lambda_\OO$ module. 
We will show that (under mild hypothesis) $\Lambda_\OO$-modules $X^{GV}_{\Q_\infty}(T,T{_\wp^+})$  and $X^{GV}_{\Q_\infty}(T_1,(T_1){_\wp^+})$ have the same characteristic ideal. Throughout this discussion, we'll assume that $X^{GV}_{\Q_\infty}(T,T{_\wp^+})$  and $X^{GV}_{\Q_\infty}(T_1,(T_1){_\wp^+})$ are $\Lambda_\OO$-torsion, as otherwise, the characteristic ideal is $0$.

To prove our result, we need an alternative description of $Sel^{GV}_{\Q_\infty}(A,A_\wp^+)$. %( \ref{selgv}).

\begin{equation}\label{eq1234}
 Sel^{GV}_{\Q_\infty}(A,A_\wp^+)= \varinjlim_{n} Sel^{GV}_{\Q_n}(A,A_\wp^+)  = \Ker\Big(H^1(\Q_\Sigma/\Q_{\infty}, A) \lra  \underset{ \ell \in \Sigma}{\prod} \, \mathcal{H}^1_{\ell}(\Q_{\infty}, A)\Big),
\end{equation}
 
where $\mathcal{H}^1_{\ell}(\Q_{\infty}, A)$ is defined as follows. For $\ell \neq p$,
\begin{equation*}
 \mathcal{H}^1_{\ell}(\Q_{\infty}, A) = \underset{ v_{\infty} \mid \ell}{\prod} H^1(G_{v_\infty}, A),
\end{equation*}
here the product is over the finite set of primes $v_{\infty} \mid \ell $ of $\Q_{\infty}$. For $\ell=p$,

\begin{equation}\label{eq12345}
 \mathcal{H}^1_{p}(\Q_{\infty}, A) = \text{Im} \Big (H^1(G_{\p_\infty},A) \to H^1(I_{\p_\infty}, A_{\wp}^-) \Big )= H^1(G_{\p_\infty},A)/ L(G_{\p_\infty},A), 
\end{equation} 
here $L(G_{\p_\infty},A)$ is the kernel of the map $H^1(G_{\p_\infty},A) \to H^1(I_{\p_\infty}, A_{\wp}^-)$. The $\Lambda_{\OO}$- submodule $ L(G_{\p_\infty},A)$ of $H^1(G_{\p_\infty},A)$ sits in the short exact sequence 
\begin{equation}\label{eq1.3}
0 \to H^1(G_{\p_\infty}, A_{\wp}^+) \to L(G_{\p_\infty},A) \to H^1_{\text{unr}}(G_{\p_\infty}, A_{\wp}^-) \to 0,
\end{equation}

where $H^1_{\text{unr}}(G_{\p_\infty}, A_{\wp}^-)$ denotes the kernel of the map $H^1(G_{\p_\infty}, A_{\wp}^-) \to H^1(I_{\p_\infty}, A_{\wp}^-).$ 

\begin{theorem}\label{muconstant}
Let $(\rho,V,V_\wp^+)$ be a triple which satisfies {\bf HYP 1, 2 \& Tor{$_\rho$}}. Assume that $H^0(\Delta_\wp,V)=0$. Then the $\Lambda_\OO$-modules $X^{GV}_{\Q_\infty}(T,T{_\wp^+})$  and $X^{GV}_{\Q_\infty}(T_1,(T_1){_\wp^+})$ have same characteristic ideal.
\end{theorem} 

\begin{proof}
By scaling a power of $\pi$, we can assume that $T \subset T_1$. Consequently, we have a $G_{\Q}$-equivariant surjective map $\phi: A \to A_1$ with finite kernel $\Phi$. Then $\Phi$ must be contained in $A[\pi^t]$ for some $t \geq 0$, where $A[\pi^t]$ denotes the $\pi^t$-torsion points. 
We also have a $G_{\Q}$-equivariant  map $\psi: A_1 \to A$ such that the composite map $\psi \circ \phi : A \to A$ is  multiplication by $\pi^t$. 
The map $\phi$ induces a map from $ Sel^{GV}_{\Q_\infty}(A,A_\wp^+)$ to  $ Sel^{GV}_{\Q_\infty}(A_1,{(A_1)}_\wp^+)$ such that the kernel is annihilated by $\pi^t$ and similarly, $\psi$ induces a map from $ Sel^{GV}_{\Q_\infty}(A_1,{(A_1)}_\wp^+)$ to $ Sel^{GV}_{\Q_\infty}(A,A_\wp^+)$. The composition map is given by multiplication by $\pi^t$. 
Therefore the characteristic ideals $C$ and $C_1$ of $X^{GV}_{\Q_\infty}(T,T{_\wp^+})$  and $X^{GV}_{\Q_\infty}(T_1,(T_1){_\wp^+})$, respectively, satisfy the relation $C_1=\pi^sC$, for some $s \in \Z$. Thus, to prove the theorem, it is enough to show that the $\mu$-invariant of $ Sel^{GV}_{\Q_\infty}(A,A_\wp^+)$ and $ Sel^{GV}_{\Q_\infty}(A_1,{(A_1)}_\wp^+)$ are the same. As the map given in equation \eqref{eq1234} is surjective (see \cite[Proposition 2.1]{gv1}), we have an exact sequence

\begin{equation*}
 0 \to  Sel^{GV}_{\Q_\infty}(A,A_\wp^+) \to H^1(\Q_\Sigma/\Q_{\infty}, A) \lra  \underset{ \ell \in \Sigma}{\prod} \, \mathcal{H}^1_{\ell}(\Q_{\infty}, A) \to 0.
\end{equation*}

This gives us $\mu \big(Sel^{GV}_{\Q_\infty}(A,A_\wp^+)\big )+ \underset{ \ell \in \Sigma}{\sum} \, \mu \big (\mathcal{H}^1_{\ell}(\Q_{\infty}, A) \big ) = \mu \big (H^1(\Q_\Sigma/\Q_{\infty}, A) \big ) $. 
Similarly we  get $\mu \big(Sel^{GV}_{\Q_\infty}(A_1,{(A_1)}_\wp^+)\big )+ \underset{ \ell \in \Sigma}{\sum} \, \mu \big (\mathcal{H}^1_{\ell}(\Q_{\infty}, A_1) \big ) = \mu \big (H^1(\Q_\Sigma/\Q_{\infty}, A_1) \big ) $. Therefore it is enough to show that 
\begin{equation}\label{eq123456}
\mu \big (H^1(\Q_\Sigma/\Q_{\infty}, A) \big )-\mu \big (H^1(\Q_\Sigma/\Q_{\infty}, A_1) \big ) =\underset{ \ell \in \Sigma}{\sum} \Big ( \, \mu \big (\mathcal{H}^1_{\ell}(\Q_{\infty}, A) \big )- \, \mu \big (\mathcal{H}^1_{\ell}(\Q_{\infty}, A_1) \big )\Big ).
\end{equation}

Since $\mathcal{H}^1_{\ell}(\Q_{\infty}, A)=0$  when $\ell$ is an archimedean prime and $\mathcal{H}^1_{\ell}(\Q_{\infty}, A)$ has a finite $\OO$-rank when $\ell \neq p$, the $\mu$-invariant is zero in both the cases. Thus we need to show that
\begin{equation}\label{eq1234567}
\mu \big (H^1(\Q_\Sigma/\Q_{\infty}, A) \big )-\mu \big (H^1(\Q_\Sigma/\Q_{\infty}, A_1) \big ) = \mu \big (\mathcal{H}^1_{p}(\Q_{\infty}, A) \big )-  \mu \big (\mathcal{H}^1_{p}(\Q_{\infty}, A_1) \big ).
\end{equation}
 
First, we will compute the value of the L.H.S of equation \eqref{eq1234567}.\\
From $\phi$, we get an exact sequence $0 \to \Phi \to A \to A_1 \to 0$ which induces  the exact sequence
\begin{equation*}
\begin{split}
H^0(\Q_\Sigma/\Q_{\infty}, A_1) \to H^1(\Q_\Sigma/\Q_{\infty}, \Phi)\to H^1(\Q_\Sigma/\Q_{\infty}, A)  \to H^1(\Q_\Sigma/\Q_{\infty}, A_1) &  \\
\to H^2(\Q_\Sigma/\Q_{\infty}, \Phi) \to H^2(\Q_\Sigma/\Q_{\infty}, A) . & 
\end{split}
\end{equation*}

The $\mu$- invariant of $H^0(\Q_\Sigma/\Q_{\infty}, A_1)$ and $H^2(\Q_\Sigma/\Q_{\infty}, A)$ are zero (see \cite[Proposition 2.8]{gv}). Therefore, we have   
\begin{equation*}
\mu \big (H^1(\Q_\Sigma/\Q_{\infty}, A) \big )-\mu \big (H^1(\Q_\Sigma/\Q_{\infty}, A_1) \big )= \mu \big (H^1(\Q_\Sigma/\Q_{\infty}, \Phi) \big )-\mu \big (H^2(\Q_\Sigma/\Q_{\infty}, \Phi) \big ).
\end{equation*}

Note that $\Phi$ is annihilated by  $\pi^t$, but we can reduce to the case where $\Phi$ is annihilated by $\pi$.
Let $\widetilde{\Lambda}_{\OO}= \Lambda_{\OO}/\pi\Lambda_{\OO} = k[[\Gamma]]$, where $k = \OO/(\pi)$. 
Then the $\mu$-invariant of $H^i(\Q_\Sigma/\Q_{\infty}, \Phi)$ is same as the $\widetilde{\Lambda}_{\OO}$-corank of $H^i(\Q_\Sigma/\Q_{\infty}, \Phi)$(See \cite[Section-2, page 5]{rs}). 
Therefore the value of $\mu \big (H^1(\Q_\Sigma/\Q_{\infty}, A) \big )-\mu \big (H^1(\Q_\Sigma/\Q_{\infty}, A_1) \big )$ is same as $\text{corank}_{\widetilde{\Lambda}_{\OO}}H^1(\Q_\Sigma/\Q_{\infty}, \Phi)- \text{corank}_{\widetilde{\Lambda}_{\OO}}H^2(\Q_\Sigma/\Q_{\infty}, \Phi)$. 
First, we will show that $H^1(\Q_\Sigma/\Q_{\infty}, \Phi)$ and $H^2(\Q_\Sigma/\Q_{\infty}, \Phi)$ are  cofinitely
generated as $\widetilde{\Lambda}_{\OO}$-modules, and then we will compute the value of
\begin{equation*}
\text{corank}_{\widetilde{\Lambda}_{\OO}}H^1(\Q_\Sigma/\Q_{\infty}, \Phi)- \text{corank}_{\widetilde{\Lambda}_{\OO}}H^2(\Q_\Sigma/\Q_{\infty}, \Phi).
\end{equation*}

For $i=1,2$ consider the restriction maps
\begin{equation*}
\rho_n^{(i)}: H^i(\Q_\Sigma/\Q_{n}, \Phi) \to H^i(\Q_\Sigma/\Q_{\infty}, \Phi)^{\Gamma_n}.
\end{equation*}
Since the cohomological dimension of $\Gal(\Q_\infty/\Q_n)$ is one, using the inflation-restriction map, we get that $\rho_n^{(1)}$ is surjective and ker$(\rho_n^{(1)}) = H^1(\Gamma_n, M)$, where $M= H^0(\Q_\Sigma/\Q_{\infty}, \Phi)$. Since $M=0$, the map $\rho_n^{(1)}$ is injective.
From Hochschild-Serre spectral sequence, we get that $\rho_n^{(2)}$ is surjective and ker$(\rho_n^{(2)}) = H^1(\Gamma_n, N)$, where $N= H^1(\Q_\Sigma/\Q_{\infty}, \Phi)$. 
The group $H^i(\Q_\Sigma/\Q, \Phi)$ is finite for all $i$. Using the surjectivity of $\rho_0^{(1)}$ and $\rho_0^{(2)}$, we get that $H^i(\Q_\Sigma/\Q_{\infty}, \Phi)^{\Gamma}$ is also finite for $i=1,2$. 
As a result, $H^1(\Q_\Sigma/\Q_{\infty}, \Phi)$ and $H^2(\Q_\Sigma/\Q_{\infty}, \Phi)$ are  cofinitely
generated as $\widetilde{\Lambda}_{\OO}$-modules.
Let $h_1$, $h_2$ denote the $\widetilde{\Lambda}_{\OO}$-corank of $H^1(\Q_\Sigma/\Q_{\infty}, \Phi)$ and $H^2(\Q_\Sigma/\Q_{\infty}, \Phi)$, respectively. 
Since $\widetilde{\Lambda}_{\OO}$ is a PID, the pontryagin dual $N^\vee$ of $N$ is the direct sum of a free $\widetilde{\Lambda}_{\OO}$-module of rank $h_1$ and a finite torsion module. 
Therefore, $N \cong N_{\text{div}} \oplus N/N_{\text{div}}$ as an $\widetilde{\Lambda}_{\OO}$-module, where $N/N_{\text{div}}$ is finite.
Since $N_{\text{div}}$ is divisible, and $H^0(\Gamma_n, N_{\text{div}})$ is finite, we have $H^1(\Gamma_n, N_{\text{div}})=0$. 
As a result, $H^1(\Gamma_n, N)=H^1(\Gamma_n, N/N_{\text{div}})$. Therefore, ker$(\rho_n^{(2)})$ is bounded by $N/N_{\text{div}}$. Since $N/N_{\text{div}}$ is finite, for sufficiently large $n$, it is fixed by  $\Gamma_n$ and consequently we get $|$ker$(\rho_n^{(2)})|=|N/N_{\text{div}}|$.

If $X$ is an $\widetilde{\Lambda}_{\OO}=k[[\Gamma]]$-module of corank $h$, then we have
\begin{equation}\label{eq12345678}
\text{dim}_k (X^{\Gamma_n})= hp^n + O(1)
 \end{equation}
as $n \to \infty$. The $k$-dimension of $X^{\Gamma_n}$ is equal to $hp^n$ for at least one $n$ if and only if $X$ is a cofree $\widetilde{\Lambda}_{\OO}$-module. By applying equation \eqref{eq12345678} to  $N$ for sufficiently large $n$, we have
\begin{equation} \label{eq1.7}
\text{dim}_k (H^1(\Q_\Sigma/\Q_{n}, \Phi))= h_1p^n +\text{dim}_k (N/N_{\text{div}}).
\end{equation}
Now from the Euler-Poincare characteristic formula for number fields we obtain 
\begin{equation} \label{eq1.8}
\text{dim}_k (H^0(\Q_\Sigma/\Q_{n}, \Phi)) + \text{dim}_k (H^2(\Q_\Sigma/\Q_{n}, \Phi))-\text{dim}_k (H^1(\Q_\Sigma/\Q_{n}, \Phi)) = -d^-(\Phi)[\Q_n:\Q].
\end{equation}
Since $H^0(\Q_\Sigma/\Q_{n}, \Phi)=0$, substituting equation \eqref{eq1.7} in \eqref{eq1.8}, for sufficiently large $n$, we get 
\begin{equation*} 
\text{dim}_k (H^2(\Q_\Sigma/\Q_{n}, \Phi))-( h_1p^n +\text{dim}_k (N/N_{\text{div}}) = -d^-(\Phi)[\Q_n:\Q]=  -d^-(\Phi)p^n,
\end{equation*}
which is same as,
\begin{equation}\label{xoxoxo}
\text{dim}_k (H^2(\Q_\Sigma/\Q_{n}, \Phi))- \text{dim}_k (N/N_{\text{div}}) = (h_1- d^-(\Phi))p^n.
\end{equation}

For sufficiently large $n$, $|$ker$\rho_n^{(2)}|=|N/N_{\text{div}}|$. Thus from equation \eqref{xoxoxo} we obtain $\text{dim}_k (H^2(\Q_\Sigma/\Q_{\infty}, \Phi)^{\Gamma_n}) = \text{dim}_k (H^2(\Q_\Sigma/\Q_{n}, \Phi))- \text{dim}_k (N/N_{\text{div}})= (h_1- d^-(\Phi))p^n$. Again using equation \eqref{eq12345678}, we obtain that the $\widetilde{\Lambda}_{\OO}$-corank of  $H^2(\Q_\Sigma/\Q_{\infty}, \Phi)$ is $h_2= h_1- d^-(\Phi).$ Therefore 
\begin{equation*}
\text{corank}_{\widetilde{\Lambda}_{\OO}}H^1(\Q_\Sigma/\Q_{\infty}, \Phi)- \text{corank}_{\widetilde{\Lambda}_{\OO}}H^2(\Q_\Sigma/\Q_{\infty}, \Phi)= d^-(\Phi).
\end{equation*}\\

Using equation \eqref{eq12345}, we identify $\mathcal{H}^1_{p}(\Q_{\infty}, A)$ with $H^1(G_{\p_\infty},A)/L(G_{\p_\infty}, A)$. Thus by  equation \eqref{eq1.3}, $\mu (L(G_{\p_\infty},A))= \mu (H^1(G_{\p_\infty},A_\wp^+)) + \mu(H^1_{\text{unr}}(G_{\p_\infty}, A_{\wp}^-))$. Therefore equation \eqref{eq1234567} becomes 
\begin{equation}\label{yoyoyo}
\begin{split}
d^-(\Phi) = (\mu (H^1(G_{\p_\infty},A))- \mu(H^1(G_{\p_\infty},A_1)))-(\mu(H^1(G_{\p_\infty}, A_{\wp}^+))-\mu(H^1(G_{\p_\infty}, {(A_1)}_{\wp}^+))) & \\
 -(\mu(H^1_{\text{unr}}(G_{\p_\infty}, A_{\wp}^-))-\mu(H^1_{\text{unr}}(G_{\p_\infty}, {(A_1)}_{\wp}^-))) .& 
\end{split}
\end{equation}

First, we will compute the value of $\mu(H^1(G_{\p_\infty},A))-\mu( H^1(G_{\p_\infty},A_1))$. Following the same method as before, one can show that $H^1(G_{\p_\infty}, \Phi)$ and $H^2(G_{\p_\infty}, \Phi)$ are cofinitely generated as $\widetilde{\Lambda}_{\OO}$-modules and the value of $\mu(H^1(G_{\p_\infty},A))-\mu( H^1(G_{\p_\infty},A_1))$  is same as $\text{corank}_{\widetilde{\Lambda}_{\OO}}H^1(G_{\p_\infty}, \Phi)- \text{corank}_{\widetilde{\Lambda}_{\OO}}H^2(G_{\p_\infty}, \Phi)$.

By inflation-restriction exact sequence, we obtain
\begin{equation*}
0\to H^1(\Gamma_{n,\p}, \Phi^{G_{\p_\infty}}) \to H^1(G_{\p_n}, \Phi) \to H^1(G_{\p_\infty}, \Phi)^{\Gamma_{n,\p}} \to H^2(\Gamma_{n,\p}, \Phi^{G_{\p_\infty}}),
\end{equation*}
where $\Gamma_{n,\p}= G_{\p_n}/G_{\p_\infty}$. As $\Phi^{G_{\p_\infty}}=0$, we obtain $\text{dim}_k (H^1(G_{\p_\infty}, \Phi)^{\Gamma_{n,\p}})= \text{dim}_k (H^1(G_{\p_n}, \Phi))$.
The $k$-dimension of $H^2(G_{\p_n}, \Phi)$ stabilize as $n$ varies. Local Euler-Poincare formula yields
\begin{equation*}
\text{dim}_k (H^1(G_{\p_\infty}, \Phi)^{\Gamma_{n,\p}})= \text{dim}_k(H^1(G_{\p_n}, \Phi)) = d(\Phi)[\Q_{n,\p_n}:\Q_p] + O(1).
\end{equation*}
Using equation \eqref{eq12345678}, we get that  $\text{corank}_{\widetilde{\Lambda}_{\OO}} (H^1(G_{\p_\infty}, \Phi))= d(\Phi) $. \\
From Hochschild-Serre spectral sequence, we get obtain the restriction map $H^2(G_{\p_n}, \Phi) \to H^2(G_{\p_\infty}, \Phi)^{\Gamma_{n,\p}}$ is surjective. As the order of the group $H^2(G_{\p_n}, \Phi)$ stabilizes as $n$ varies by using equation \eqref{eq12345678}, we obtain $\text{corank}_{\widetilde{\Lambda}_{\OO}} (H^2(G_{\p_\infty}, \Phi))= 0 $. 
Hence, $\mu(H^1(G_{\p_\infty},A))-\mu( H^1(G_{\p_\infty},A_1) )=d(\Phi)$. 

Similarly one can show that $\mu(H^1(G_{\p_\infty},A_\wp^+))-\mu( H^1(G_{\p_\infty},(A_1)_\wp^+))=d^+(\Phi)$. As a consequence, equation \eqref{yoyoyo} becomes
 \begin{equation}\label{zozozo}
 \mu(H^1_{\text{unr}}(G_{\p_\infty}, A_{\wp}^-))- \mu(H^1_{\text{unr}}(G_{\p_\infty}, {(A_1)}_{\wp}^-))=0.
\end{equation}
Recall that $H^1_{\text{unr}}(G_{\p_\infty}, A_{\wp}^-) = \Ker(H^1(G_{\p_\infty}, A_{\wp}^-)  \to H^1(I_{\p_\infty}, A_{\wp}^-))$. Since $G_{\p_\infty}/I_{\p_\infty}$ is topologically cyclic, we obtain 
\begin{equation*}
0 \to  H^1(G_{\p_\infty}/I_{\p_\infty}, (A_{\wp}^-)^{I_{\p_\infty}}) \to H^1(G_{\p_\infty}, A_{\wp}^-)  \to H^1(I_{\p_\infty}, A_{\wp}^-) \to 0
\end{equation*}
and moreover $\text{corank}(H^1(G_{\p_\infty}/I_{\p_\infty}, (A_{\wp}^-)^{I_{\p_\infty}})) = \text{corank}(H^0(G_{\p_\infty}/I_{\p_\infty}, (A_{\wp}^-)^{I_{\p_\infty}})$. 
By our assumption $H^0(\Delta_\wp, V)=0$ and  Lemma \ref{lemma2.1}, we have $H^0(G_{\p_\infty}, V_\wp^-)=0$. Since the order of  $\Delta$ is co-prime to  $p$, therefore $H^0(G_{\p_\infty}, V_{\wp}^-)=0$ implies that $H^0(G_{\p_\infty}/I_{\p_\infty}, (A_{\wp}^-)^{I_{\p_\infty}})= (A_{\wp}^-)^{G_{\p_\infty}}=0$. It follows that $\text{corank}(H^1(G_{\p_\infty}/I_{\p_\infty}, (A_{\wp}^-)^{I_{\p_\infty}}))=0$ and therefore $\mu(H^1_{\text{unr}}(G_{\p_\infty}, A_{\wp}^-))=0$. 
A similar argument shows that $\mu(H^1_{\text{unr}}(G_{\p_\infty}, {(A_1)}_{\wp}^-))=0$.

\end{proof}

\section{Algebraic Functional Equation}

In this section, we construct an algebraic functional equation relating characteristic ideals of the dual Selmer groups (provided it is torsion) associated to Artin representations $V$ and $U$. 

An essential ingredient for our proof is a version of Cassels-Tate pairing of Flach \cite{fl}.

\begin{proposition}[{\bf Generalised Cassels-Tate pairing of Flach}]\label{4.1}
Let $(\rho,V,V_\wp^+)$ be a triple which satisfies {\bf HYP 1 \& 2}. Assume that $H^0(\Delta_\wp,V) = H^0(\Delta_\wp, U)=0$. Let $\varphi: \Gamma \to \OO^*$ be a continuous group homomorphism.  Then for all $n$, we have a  perfect pairing 
\begin{equation*}
 <,>:\frac{Sel^{g}_{\Q_n}(A,{A}_\wp^+,\varphi^{-1})}{Sel^{g}_{\Q_n}(A,{A}_\wp^+,\varphi^{-1})_{\mathrm{div}}} \times \frac{Sel^{g}_{\Q_n}(A^*,{(A^*)}_\wp^+,\varphi)}{Sel^{g}_{\Q_n}(A^*,{(A^*)}_\wp^+,\varphi)_{\mathrm{div}}} \longrightarrow \Q_p /\Z_p .
\end{equation*}

In particular, if the Selmer groups $Sel^{g}_{\Q_n}(A,{A}_\wp^+,\varphi^{-1})$ and $Sel^{g}_{\Q_n}(A^*,{(A^*)}_\wp^+,\varphi)$ are finite, then $Sel^{g}_{\Q_n}(A^*,{(A^*)}_\wp^+,\varphi)$  is the Pontryagin dual of $Sel^{g}_{\Q_n}(A,{A}_\wp^+,\varphi^{-1})$ under the above paring.

\end{proposition}

\begin{proof}
From Lemma \ref{3.3}, we obtained that $H^1_g(G_{v_n},  V(\varphi^{-1})^*)$ is the orthogonal complement of  $H^1_g(G_{v_n}, V(\varphi^{-1}))$ under the local Tate pairing 
$$H^1(G_{v_n}, V(\varphi^{-1})) \times H^1(G_{v_n}, V(\varphi^{-1})^*) \longrightarrow \Q_p/\Z_p,$$
for each place $v_n$ of $\Q_n$ such that $v_n \mid \Sigma$. Then from the generalized Cassels-Tate pairing of Flach \cite[Theorem 1]{fl}, we obtain a perfect pairing 
\begin{equation}\label{CT}
 <,>:\frac{Sel^{g}_{\Q_n}(A,{A}_\wp^+,\varphi^{-1})}{Sel^{g}_{\Q_n}(A,{A}_\wp^+,\varphi^{-1})_{\mathrm{div}}} \times \frac{Sel^{g}_{\Q_n}(A^*,{(A^*)}_\wp^+,\varphi)}{Sel^{g}_{\Q_n}(A^*,{(A^*)}_\wp^+,\varphi)_{\mathrm{div}}} \longrightarrow \Q_p /\Z_p .
 \end{equation}

If $Sel^{g}_{\Q_n}(A,{A}_\wp^+,\varphi^{-1})$ and $Sel^{g}_{\Q_n}(A^*,{(A^*)}_\wp^+,\varphi)$ are finite, then $Sel^{g}_{\Q_n}(A,{A}_\wp^+,\varphi^{-1})_{\mathrm{div}}$ and $Sel^{g}_{\Q_n}(A^*,{(A^*)}_\wp^+,\varphi)_{\mathrm{div}}$ are zero. Hence the pairing in equation  \eqref{CT} gives a perfect pairing between $Sel^{g}_{\Q_n}(A,{A}_\wp^+,\varphi^{-1})$ and $Sel^{g}_{\Q_n}(A^*,{(A^*)}_\wp^+,\varphi)$. 

\end{proof}
Before Proceeding to the proof of the functional equation, we introduce a few notations following \cite{pr}.\\

Let $\OO$ be the ring of integers for a finite extension of $\Q_p$ and $M$ is a finitely generated $\Lambda_\OO$-module. Then Perrin-Riou \cite[Page 728]{pr} defines, $$a^{i}_{\Lambda_\OO}(M) := \mathrm{Ext}^{i}_{\Lambda_\OO}(M, \Lambda_\OO )  \text{\quad} \mathrm{for} \text{\quad} i \geq 0.$$

For a $\Lambda_\OO$-module $M$, $M^{\iota}$ denotes the same module $M$ but the action of $\Gamma$ is changed via the involution sending $\gamma \to \gamma^{-1}$ for every $\gamma \in \Gamma$.

\begin{lemma} \label{4.2}
Let $M$ be a finitely generated torsion $\Lambda_\OO$-module such that $M_{\Gamma_n}$ is finite for each $n$, then $$ a^{1}_{\Lambda_\OO}(M) \cong \varprojlim_n ({M^{\iota}}^{\vee})^{\Gamma_n}.$$
\end{lemma}

\begin{proof}
This is well known. A proof can be found in \cite[Proposition 1.3.1]{pr}.
\end{proof}

\begin{lemma} \label{4.3}
Let $M$ be a finitely generated torsion $\Lambda_\OO$-module. Then $$ C_{\Lambda_\OO}(M)= C_{\Lambda_\OO}(a^{1}_{\Lambda_\OO}(M)),$$ as ideals in ${\Lambda_\OO}.$
\end{lemma}

\begin{proof}
This is well known. A proof can be found in \cite[Lemma 3.5]{jp}.
\end{proof}

\begin{lemma} \label{4.4}
Let $M$ and $N$ be two finitely generated torsion $\Lambda_\OO$-modules. Then there exists a character $\varphi: \Gamma \to \OO^*$ such that $M(\varphi)_{\Gamma_n}$ and $N(\varphi^{-1})_{\Gamma_n}$ are finite for all $n$.
\end{lemma}
\begin{proof}
This is well known. A proof can be found in \cite[Lemma 3.6]{jp}.
\end{proof}

Following Rubin, for a character $\varphi: \Gamma \to \OO^*$, we define an $\OO$-linear automorphism $Tw_{\varphi} : \Lambda_\OO \to \Lambda_\OO$ induced by the map $Tw_{\varphi}(\gamma)= \varphi(\gamma)\gamma$. For a finitely generated torsion $\Lambda_\OO$-module $M$, $Tw_{\varphi}(M)$ denotes the same module $M$ but the action of $\Gamma$ is changed via $Tw_{\varphi}(\gamma)$ for every $\gamma \in \Gamma$.

\begin{lemma}\label{twistofideal}
Let $M$ be a finitely generated torsion $\Lambda_\OO$-module. Then for any character $\varphi: \Gamma \to \OO^\times$ we have
\begin{equation}\label{eq4.26}
Tw_{\varphi} (C_{\Lambda_\OO} (M \otimes \varphi)) = C_{\Lambda_\OO} (M).
\end{equation} 

\end{lemma}

\begin{proof}
This is well known. A proof can be found in \cite[Chap.VI, Lemma-1.2]{ru}.
\end{proof}

\begin{rem}\label{fingen2}
Let $(\rho,V,V_\wp^+)$ be a triple which satisfies {\bf HYP 1 \& 2}. Then for $\dagger \in \{g, GV \}$,  $X^{\dagger}_{\Q_\infty}(T,T{_\wp^+})$ is a finitely generated $\Lambda_\OO$-module by Lemma \ref{fingen}. Since $X^{\dagger}_{\Q_\infty}(T,T{_\wp^+},\varphi) \cong X^{\dagger}_{\Q_\infty}(T,T{_\wp^+}) \otimes_\OO \OO(\varphi^{-1})$, it follows that $X^{\dagger}_{\Q_\infty}(T,T{_\wp^+},\varphi)$ is also a finitely generated $\Lambda_\OO$-module. Moreover, if $X^{\dagger}_{\Q_\infty}(T,T{_\wp^+})$ is $\Lambda_\OO$-torsion, then $X^{\dagger}_{\Q_\infty}(T,T{_\wp^+}, \varphi)$ is also $\Lambda_\OO$-torsion.\\
Recall that $V^*= U(\kappa)$ for some Artin representaion $(\sigma,U)$ satisfying {\bf HYP 1 \& 2} and $V(\varphi)^*= V^*(\varphi^{-1})= U(\varphi^{-1}\kappa)$. Thus, it follows that $X^{\dagger}_{\Q_\infty}(T',(T'){_\wp^+}, \varphi)$ and $X^{\dagger}_{\Q_\infty}(T^*,(T^*){_\wp^+},\varphi)$ are finitely generated $\Lambda_\OO$-module. Moreover, if $X^{\dagger}_{\Q_\infty}(T',(T'){_\wp^+})$ is $\Lambda_\OO$-torsion, then $X^{\dagger}_{\Q_\infty}(T',(T'){_\wp^+}, \varphi)$  and $X^{\dagger}_{\Q_\infty}(T^*,(T^*){_\wp^+},\varphi)$ are also $\Lambda_\OO$-torsion.\\
\end{rem}

Now we are ready to prove our main theorem. Since the theorem relates characteristic ideals of the dual Selmer groups, it is a necessary assumption that the dual Selmer groups are $\Lambda$-torsion. Proof of our algebraic functional equation follows an argument similar to arguments provided in \cite{pr,jp,jm,jo}.\\

\begin{theorem}\label{4.7}
Let $(\rho,V,V_\wp^+)$ be a triple that satisfies  {\bf HYP 1, 2 \& Tor{$_\rho$}}, and $T$ be an $\OO$-lattice of $V$ invariant under the action of $\Delta$. Let $(\sigma,U)$ denote the Artin representations $U:= \Hom_{\FF}(V, \FF(\tau))$, where $\tau$ denotes the Teichmuller character and $T^\prime$ the corresponding lattice in $U$. Assume that $(\sigma, U, U_\wp^+)$ satisfies {\bf HYP Tor{$_\sigma$}}. Further, assume that the $\Delta_\wp$-representations $V$ and $U$ do not contain trivial sub-representation. Then the functional equation holds for the dual Selmer group associated to $(\rho,V,V_\wp^+)$, that is we have the following equality of ideals in ${\Lambda_\OO}$
\begin{equation*}
  C_{\Lambda_\OO}(X^{GV}_{\Q_\infty}(T,T{_\wp^+})^\iota)= C_{\Lambda_\OO}(X^{GV}_{\Q_\infty}(T^*,(T^*){_\wp^+})) = Tw_{\kappa} (C_{\Lambda_\OO} (X^{GV}_{\Q_\infty}(T',(T'){_\wp^+})),
\end{equation*}
where $\kappa: \Gamma \to 1+ p\Z_p$ is an isomorphism coming from the action of  $\Gal(K_{\infty}/K)$ on $\mu_{p^{\infty}}$.

\end{theorem}

\begin{proof}  
Under the assumption  {\bf HYP Tor{$_\rho$}} and {\bf HYP Tor{$_\sigma$}}, and by Lemma \ref{4.4}, we can get a $\varphi$ such that $(X^{GV}_{\Q_\infty}(T,T{_\wp^+},\varphi^{-1}))_{\Gamma_n}$ and $(X^{GV}_{\Q_\infty}(T^*,(T^*){_\wp^+},\varphi))_{\Gamma_n}$ are finite for all $n$. Thus by Theorem \ref{3.1} and Corollary \ref{3.2}, the groups $Sel^{GV}_{\Q_n}(A,{A}_\wp^+,\varphi^{-1})$ and $Sel^{GV}_{\Q_n}(A^*,{(A^*)}_\wp^+,\varphi)$ are finite, respectively for all $n$.
Now by Theorem \ref{2.1} and Corollary \ref{2.3}, $Sel^{g}_{\Q_n}(A,{A}_\wp^+,\varphi^{-1})$ and $Sel^{g}_{\Q_n}(A^*,{(A^*)}_\wp^+,\varphi)$ are finite for all $n$.
 Then by generalized Cassels-Tate pairing of Flach (Corollary \ref{4.1}), we have $(Sel^{g}_{\Q_n}(A^*,{(A^*)}_\wp^+,\varphi))^\vee \cong Sel^{g}_{\Q_n}(A,{A}_\wp^+,\varphi^{-1})$ for all $n$. 
From Corollary \ref{2.3} (respectively, Theorem \ref{2.1}) the kernel and the cokernel of the map $(Sel^{GV}_{\Q_n}(A^*,{(A^*)}_\wp^+,\varphi))^\vee \to (Sel^{g}_{\Q_n}(A^*,{(A^*)}_\wp^+,\varphi))^\vee$ (respectively, $Sel^{g}_{\Q_n}(A,{A}_\wp^+,\varphi^{-1}) \to Sel^{GV}_{\Q_n}(A,{A}_\wp^+,\varphi^{-1})$), is finite and independent on $n$. 
 From Theorem \ref{3.1}, the restriction map $Sel^{GV}_{\Q_n}(A,{A}_\wp^+,\varphi^{-1})  \to  Sel^{GV}_{\Q_\infty}(A,{A}_\wp^+,\varphi^{-1})^{\Gamma_{n}}$ is injective, and the cardinality of cokernel is finite and independent on $n$. Now by combining all the above results, we get a map
\begin{equation*}
\begin{split}
	\eta_n: (Sel^{GV}_{\Q_n}(A^*,{(A^*)}_\wp^+,\varphi))^\vee \to (Sel^{g}_{\Q_n}(A^*,{(A^*)}_\wp^+,\varphi))^\vee 
 \cong Sel^{g}_{\Q_n}(A,{A}_\wp^+,\varphi^{-1})   & \\ \to Sel^{GV}_{\Q_n}(A,{A}_\wp^+,\varphi^{-1})  \to  Sel^{GV}_{\Q_\infty}(A,{A}_\wp^+,\varphi^{-1})^{\Gamma_{n}},
\end{split}
\end{equation*}

which is a pseudo-isomorphism of an $\Lambda_{\OO}$-modules. 

Now by taking the inverse limit of the maps $\eta_n$, we get a pseudo-isomorphism
\begin{equation}\label{4.28}
 \eta: X^{GV}_{\Q_\infty}(T^*,(T^*){_\wp^+},\varphi) \to \varprojlim_n Sel^{GV}_{\Q_\infty}(A,{A}_\wp^+,\varphi^{-1})^{\Gamma_{n}},
\end{equation}
as an ${\Lambda_\OO}$-modules.

By applying Lemma \ref{4.2} to $X^{GV}_{\Q_\infty}(T,T_\wp^+,\varphi^{-1})^{\iota}$ , we get

\begin{equation}\label{4.29}
\varprojlim_n Sel^{GV}_{\Q_\infty}(A,{A}_\wp^+,\varphi^{-1})^{\Gamma_{n}} = \varprojlim_n ({X^{GV}_{\Q_\infty}(T,T_\wp^+,\varphi^{-1})^\vee )}^{\Gamma_{n}}= a^{1}_{\Lambda_\OO}(X^{GV}_{\Q_\infty}(T,T_\wp^+,\varphi^{-1})^{\iota}).
\end{equation}
From equations \eqref{4.28} and \eqref{4.29}, we have an ${\Lambda_\OO}$-modules pseudo-isomorphism 
\begin{equation}\label{4.30}
  X^{GV}_{\Q_\infty}(T^*,(T^*){_\wp^+},\varphi) \to a^{1}_{\Lambda_\OO}(X^{GV}_{\Q_\infty}(T,T_\wp^+,\varphi^{-1})^{\iota}).
\end{equation}

Recall that $X^{GV}_{\Q_\infty}(T^*,(T^*){_\wp^+}) = X^{GV}_{\Q_\infty}(T^*,(T^*){_\wp^+}) \otimes \varphi^{-1} $ and $a^{1}_{\Lambda_\OO}(X^{GV}_{\Q_\infty}(T,T_\wp^+,\varphi^{-1})^{\iota}) \cong a^{1}_{\Lambda_\OO}(X^{GV}_{\Q_\infty}(T,T_\wp^+)^{\iota}) \otimes \varphi^{-1}$.  By tensoring $\varphi$ with equation \eqref{4.30}, we get  a pseduo-isomorphsim

\begin{equation*}
  X^{GV}_{\Q_\infty}(T^*,(T^*){_\wp^+}) \to a^{1}_{\Lambda_\OO}(X^{GV}_{\Q_\infty}(T,(T){_\wp^+})^{\iota}).
\end{equation*}

Hence $C_{\Lambda_\OO}( X^{GV}_{\Q_\infty}(T^*,(T^*){_\wp^+}))=C_{\Lambda_\OO}(X^{GV}_{\Q_\infty}(T,(T){_\wp^+})^{\iota})$. By Lemma \ref{4.3}, we obtain  
\begin{equation}\label{4.31}
C_{\Lambda_\OO}( X^{GV}_{\Q_\infty}(T^*,(T^*){_\wp^+}))= C_{\Lambda_\OO}(X^{GV}_{\Q_\infty}(T, T_\wp^+)^{\iota}).
\end{equation}

From the relation $V^* \cong U \otimes_{\FF} \FF(\kappa)$, and by taking $M=X^{GV}_{\Q_\infty}(T^*,(T^*){_\wp^+})$ and $\varphi = \kappa$ in lemma \eqref{twistofideal}, we get
\begin{equation}\label{4.32}
 Tw_{\kappa} (C_{\Lambda_\OO}(X^{GV}_{\Q_\infty}(T',(T')_\wp^+)) = C_{\Lambda_\OO} (X^{GV}_{\Q_\infty}(T^*,(T^*){_\wp^+})).
\end{equation}

From equations \eqref{4.31} and \eqref{4.32}, we get the required results.

\end{proof}

\begin{corollary}
Let $(\rho,V,V_\wp^+)$ be a triple which satisfies  {\bf HYP 1, 2 \& Tor{$_\rho$}}, and $T$ be an $\OO$-lattice of $V$ invariant under the action of $\Delta$. Let $(\sigma,U)$ denote the Artin representations $U:= \Hom_{\FF}(V, \FF(\tau))$, where $\tau$ denotes the Teichmuller character and $T^\prime$ the corresponding lattice in $U$. Assume that $(\sigma, U, U_\wp^+)$ satisfies {\bf HYP Tor{$_\sigma$}}. Further, assume that the $\Delta_\wp$-representations $V$ and $U$ do not contain trivial sub-representation. Then the functional equation holds for the dual Selmer group $X^{g}_{\Q_\infty}(T,T{_\wp^+})$ associated to $(\rho,V,V_\wp^+)$, that is we have the following equality of ideals in ${\Lambda_\OO}$
\begin{equation}
  C_{\Lambda_\OO}(X^{g}_{\Q_\infty}(T,T{_\wp^+})^\iota)= C_{\Lambda_\OO}(X^{g}_{\Q_\infty}(T^*,(T^*){_\wp^+})) = Tw_{\kappa} (C_{\Lambda_\OO} (X^{g}_{\Q_\infty}(T',(T'){_\wp^+})),
\end{equation}
here $\kappa: \Gamma \to 1+ p\Z_p$ is an isomorphism coming from the action of  $\Gal(K_{\infty}/K)$ on $\mu_{p^{\infty}}$.
\end{corollary}

\begin{proof}
The proof is similar to the proof of Theorem \ref{4.7}.

\end{proof}

\section{Examples}
In this section, we provide examples of Artin representations that satisfy the conditions of Theorem B. We took the help of the computer algebra system SAGE \cite{sage}.\\

\begin{exmp}
Let us fix a prime $p=7$ and an embedding $\iota_7$ of $\overline{\Q}$ into $\overline{\Q}_7$. Let $K_1$ be the splitting field of $f(x)= x^3+2x+1$ over $\Q$. The Galois group  $\Gal (K_1/ \Q) \cong S_3$. Let $\rho_1: \Gal (K_1/ \Q)  \longrightarrow \gl_2(\C)$ be the standard representation of $S_3$. Let $K= K_1(\zeta_7)$ and $\rho: \Delta \to \gl_2(\C)$ be the representation $\rho = \rho_1 \circ \pi$, where $\pi: \Delta =  \Gal(K/\Q) \to \Gal(K_1/\Q)$ is the natural map. Using the embedding $\iota_7$,  we can think $\rho$ as an irreducible representation  $\rho: \Delta \to \gl_\FF(V)$ where $\FF = \Q_7(\zeta_3) = \Q_7$ and $V$ is a two dimensional vector space over $\Q_7$. The representation $\sigma$ is defined as $(\sigma,U) = \Hom_{\Q_7}(V,\Q_7(\tau))$, here $\tau$ is the Teichmuller character.\\

For any Archimedian place $v$ of $K$,  $\Delta_v = \Gal(K_{v}/\R) \cong \Z/2\Z$. The complex conjugation, when seen as an element of $S_3$, can be identified with a transposition. Therefore, the multiplicity of the trivial representation in $\rho|_{\Delta_v}$ is $1$, that is, $d^+(\rho)= d^-(\rho)=1$. As a consequence, we have $d^+(\sigma) = d^-(\sigma)=1$. \\
Using \cite{sage}, we get $|\Delta_\p'| = 3$ for any prime $\p$ above $7$, here $\Delta_\p'$ denotes the decomposition group of a prime $\p$ in $K_1$ above $7$. Now $\rho|_{\Delta_\wp} \cong \Q_7(\chi_1) + \Q_7 (\chi_1 ^{-1})$, where $\chi_1 : \Delta_\wp \to \Delta_\wp' \cong \Z/3\Z \longrightarrow \C^{*}$ is the character which sends $a \in (\Z/3\Z)^\times$ to $\zeta_3^a$. Hence $\rho|_{\Delta_\wp}$ does not contain trivial sub-representation. We see that $\sigma|_{\Delta_\wp} = \Q_7 (\tau \chi_1 ^{-1})+ \Q_7(\tau\chi_1)$ and hence $\sigma|_{\Delta_\wp}$ does not contain trivial sub-representation. \\

The triple $(\rho, V, \Q_7(\chi_1))$ and $(\sigma, U , \Q_7(\tau \chi_1^{-1}))$ satisfy  {\bf HYP 1 \& 2}. By  \cite[Proposition 4.7]{gv}, the triples satisfy {\bf HYP Tor{$_\rho$}} and {\bf HYP Tor{$_\sigma$}} as $d^+(\rho) = d^+(\sigma)=1$. Hence all the conditions of {\bf Theorem B} are satisfied for this triple.

%Thus the representation $\rho$ satisfies all the hypothesis of Theorem \ref{4.7}. Now we can apply Theorem \ref{4.7} to the triple $(p,V,V_\wp^+)$ for $V_\wp^+ \in \{\FF(\chi_1),\FF (\chi_1 ^{-1}) \}$ .\\

\end{exmp}

\begin{exmp}
Let us fix a prime $p=7$ and an embedding $\iota_7$ of $\overline{\Q}$ into $\overline{\Q}_7$. Let $K_1$ be the splitting field of $f(x)= x^4-x-1$ over $\Q$ and $K= K_1(\zeta_7)$. Let $\rho_1: \Gal(K_1/\Q) \cong S_4 \to \gl_2(\C)$ be the $2$-dimensional irreducible representation of $S_4$ and $\rho= \rho_1 \circ \pi : \Gal(K/\Q) \to \Gal(K_1/\Q) \to \gl_2(\C)$ be the lift of $\rho_1$. Via $\iota_7$, we can identify $\rho: \Delta \to \gl_{\FF}(V)$. In this case, we have $d^+(\rho)=d^+(\sigma)=1$. Moreover,  $\rho|_{\Delta_\wp} \cong \FF(\chi_1) + \FF (\chi_1 ^{-1})$ and $\sigma|_{\Delta_\wp} = \FF (\tau \chi_1 ^{-1})+ \FF (\tau\chi_1)$. \\
The triple $(\rho, V, \Q_7(\chi_1))$ and $(\sigma, U , \Q_7(\tau \chi_1^{-1}))$ satisfy  all the conditions of {\bf Theorem B}.
\end{exmp}

\begin{exmp}
Let us fix a prime $p=17$ and an embedding $\iota_{17}$ of $\overline{\Q}$ into $\overline{\Q}_{17}$. Let $K_1$ be the splitting field of $f(x)= x^5-5x^2-3$ over $\Q$ and $K= K_1(\zeta_{17})$. The Galois group $\Delta' =\Gal (K_1/ \Q) \cong D_5 \cong \big < r,s \,\,|\,\,r^5=s^2=1,\,srs=r^{-1}\big >$. Let $\rho_1: \Gal(K_1/\Q) \cong D_5 \to \gl_2(\C)$ be the $2$-dimensional irreducible representation of $D_5$ whose character $\chi$ takes value $\chi(r) = 2 \cos(2 \pi/5)$ and $\rho= \rho_1 \circ \pi : \Gal(K/\Q) \to \Gal(K_1/\Q) \to \gl_2(\C)$ be the lift of $\rho_1$. Via $\iota_{17}$, we can identify $\rho: \Delta \to \gl_{\FF}(V)$. In this case, we can identify the complex conjugation as a product of transposition, and we have $d^+(\rho)=d^+(\sigma)=1$. Moreover,  $\rho|_{\Delta_\wp} \cong \FF(\chi_1) + \FF (\chi_1 ^{-1})$ and $\sigma|_{\Delta_\wp} = \FF (\tau \chi_1 ^{-1})+ \FF(\tau\chi_1)$, here $\chi_1 : \Delta_\wp \to \Delta_\wp' \cong \Z/5\Z \longrightarrow \C^{*}$ is the character which sends $a \in (\Z/5\Z)^\times$ to $\zeta_5^a$. \\
The triple $(\rho, V, \FF(\chi_1))$ and $(\sigma, U , \FF(\tau \chi_1^{-1}))$ satisfy  all the conditions of {\bf Theorem B}.

\end{exmp}

\begin{exmp}
Let us fix a prime $p=11$ and an embedding $\iota_{11}$ of $\overline{\Q}$ into $\overline{\Q}_{11}$. Let $K_1$ be the splitting field of $f(x)= x^5-2$ over $\Q$ and $K= K_1(\zeta_{11})$. The Galois group $\Delta' =\Gal (K_1/ \Q) \cong F_{20} \cong \big < a,b\,|\, a^5=1, b^4=1,bab^{-1}=a^{3}\big >$. Let $\rho_1: \Gal(K_1/\Q) \to \gl_4(\C)$ be the $4$-dimensional irreducible representation of $F_{20}$ and $\rho= \rho_1 \circ \pi : \Gal(K/\Q) \to \Gal(K_1/\Q) \to \gl_4(\C)$ be the lift of $\rho_1$. Via $\iota_{11}$, we can identify $\rho: \Delta \to \gl_{\FF}(V)$. In this case, we can identify the complex conjugation as a product of transposition, and we have $d^+(\rho)=d^+(\sigma)=2$. Moreover,  $\rho|_{\Delta_\wp} \cong \FF(\chi_1) + \FF (\chi_1 ^2) + \FF (\chi_1 ^3) + \FF (\chi_1 ^4)$ and $\sigma|_{\Delta_\wp} = \FF (\tau \chi_1 ^4)+ \FF(\tau\chi_1^3) +  \FF(\tau\chi_1^2) +  \FF(\tau\chi_1)$, here $\chi_1 : \Delta_\wp \to \Delta_\wp' \cong \Z/5\Z \longrightarrow \C^{*}$ is the character which sends $a \in (\Z/5\Z)^\times$ to $\zeta_5^a$. \\
The triple $(\rho, V, \FF(\chi_1) +  \FF(\chi_1^2))$ and $(\sigma, U , \FF(\tau \chi_1^4)+  \FF(\tau\chi_1^3))$ satisfy  all the conditions of {\bf Theorem B} if we assume that the dual Selmer groups are torsion.

\end{exmp}

\begin{exmp}
Let us fix a prime $p=11$ and an embedding $\iota_{11}$ of $\overline{\Q}$ into $\overline{\Q}_{11}$. Let $K_1$ be the splitting field of $f(x)= x^5-x^2-2x-3$ over $\Q$ and $K= K_1(\zeta_{11})$. Let $\rho_1: \Gal(K_1/\Q) \cong A_5 \to \gl_4(\C)$ be the $4$-dimensional irreducible representation of $A_{5}$ and $\rho= \rho_1 \circ \pi : \Gal(K/\Q) \to \Gal(K_1/\Q) \to \gl_4(\C)$ be the lift of $\rho_1$. Via $\iota_{11}$, we can identify $\rho: \Delta \to \gl_{\FF}(V)$. In this case, we can identify the complex conjugation as a product of transposition, and we have $d^+(\rho)=d^+(\sigma)=2$. Moreover,  $\rho|_{\Delta_\wp} \cong \FF(\chi_1) + \FF (\chi_1 ^2) + \FF (\chi_1 ^3) + \FF (\chi_1 ^4)$ and $\sigma|_{\Delta_\wp} = \FF (\tau \chi_1 ^4)+ \FF(\tau\chi_1^3) +  \FF(\tau\chi_1^2) +  \FF(\tau\chi_1)$, here $\chi_1 : \Delta_\wp \to \Delta_\wp' \cong \Z/5\Z \longrightarrow \C^{*}$ is the character which sends $a \in (\Z/5\Z)^\times$ to $\zeta_5^a$. \\
The triple $(\rho, V, \FF(\chi_1) +  \FF(\chi_1^2))$ and $(\sigma, U , \FF(\tau \chi_1^4)+  \FF(\tau\chi_1^3))$ satisfy  all the conditions of {\bf Theorem B} if we assume that the dual Selmer groups are torsion.

\end{exmp}

\begin{exmp}
Let us fix a prime $p=29$ and an embedding $\iota_{29}$ of $\overline{\Q}$ into $\overline{\Q}_{29}$. Let $K_1$ be the splitting field of $f(x)= x^5-x^2-2x-3$ over $\Q$ and $K= K_1(\zeta_{29})$. Let $\rho_1: \Gal(K_1/\Q) \cong A_5 \to \gl_3(\C)$ be a $3$-dimensional irreducible representation of $A_{5}$ and $\rho= \rho_1 \circ \pi : \Gal(K/\Q) \to \Gal(K_1/\Q) \to \gl_3(\C)$ be the lift of $\rho_1$. Via $\iota_{29}$, we can identify $\rho: \Delta \to \gl_{\FF}(V)$. In this case, we can identify the complex conjugation as a product of transposition, and we have $d^+(\rho)=d^-(\sigma)=1$, $d^-(\rho)=d^+(\sigma)=2$. Moreover,  $\rho|_{\Delta_\wp} \cong \FF(\chi) + \delta $ and $\sigma|_{\Delta_\wp} = \FF (\tau \chi)+ \delta\tau$, here $\chi : \Delta_\wp \to \Delta_\wp' \cong S_3 \longrightarrow \C^{*}$ is the sign character and  $\delta : \Delta_\wp \to \Delta_\wp' \cong S_3 \longrightarrow \gl_2(\C)$ is the standard representation. \\
The triple $(\rho, V, \FF(\chi))$ and $(\sigma, U , \delta\tau ) $ satisfy  all the conditions of {\bf Theorem B}.

\end{exmp}

\vskip 5mm

\noindent {\bf Dipramit Majumdar}, Department of Mathematics, Indian
Institute of Technology Madras, IIT P.O.
Chennai 600036, India (dipramit@gmail.com) \\

\noindent {\bf Subhasis Panda}, Department of Mathematics, Indian
Institute of Technology Madras, IIT P.O.
Chennai 600036, India (subhasispanda559@gmail.com ) \\

\end{document}